\documentclass[oneside]{amsart}

\usepackage[utf8]{inputenc}
\usepackage{svg}
\usepackage{anyfontsize}
\usepackage{hyperref}
\usepackage{geometry}

\usepackage{amsmath}
\usepackage{amssymb}
\usepackage{amsthm}
\usepackage{graphicx}
\usepackage{blindtext}
\usepackage[textwidth=30mm]{todonotes}
\usepackage{cite}	

\usepackage{booktabs}
\usepackage{multirow}

\usepackage{algorithm}
\usepackage{algpseudocode}
\usepackage{dsfont}
\usepackage{bbm}
\usepackage{subcaption} 

\usepackage{pgf}
\usepackage{tikz}
\usepackage{xcolor}
\usepackage{mathtools}
\usepackage{appendix}

\usepackage{siunitx}  

\theoremstyle{definition}
\newtheorem{theorem}{Theorem}[section]
\newtheorem{lemma}[theorem]{Lemma}
\newtheorem{proposition}[theorem]{Proposition}

\newtheorem{remark}[theorem]{Remark}
\newtheorem{example}{Example}[section]

\usepackage[nameinlink]{cleveref}
\crefname{equation}{}{}
\crefname{proposition}{Proposition}{Propositions}
\crefname{theorem}{Theorem}{Theorems}
\crefname{lemma}{Lemma}{Lemmas}
\crefname{corollary}{Corollary}{Corollaries}
\crefname{remark}{Remark}{Remarks}
\crefname{section}{Section}{Sections}
\crefname{figure}{Figure}{Figures}
\crefname{algorithm}{Algorithm}{Algorithms}
\crefname{example}{Example}{Examples}
\crefname{table}{Table}{Tables}

\DeclareMathOperator{\trace}{tr}
\DeclareMathOperator{\rank}{rank} 
\DeclareMathOperator{\sign}{sign}
\DeclareMathOperator{\diag}{diag} 
\DeclareMathOperator{\vspan}{span} 	

\newcommand{\R}{\mathbb{R}} 
\newcommand{\C}{\mathbb{C}} 
\newcommand{\N}{\mathbb{N}} 
\newcommand{\prob}{\mathbb{P}}
\newcommand{\expec}{\mathbb{E}}

\renewcommand{\vec}[1]{\boldsymbol{#1}}		
\newcommand{\de}{\mathrm{d}} 

\newcommand{\kryl}{\mathcal{K}}		
\newcommand{\rat}{\mathcal{Q}}		

\newcommand{\dsum}{\displaystyle\sum}

\DeclarePairedDelimiter{\abs}{\lvert}{\rvert}
\DeclarePairedDelimiter{\norm}{\lVert}{\rVert}

\newcommand{\ceil}[1]{\left \lceil #1 \right \rceil }

\newcommand{\card}{\#}	
\DeclareMathOperator{\eigcount}{n_e}	
\newcommand{\quadform}{\texttt{q}}		
\newcommand{\nnz}{\texttt{nnz}}			
\newcommand{\postbound}{\texttt{bound}}	
\newcommand{\qfupper}{\texttt{q}^{\text{U}}}	
\newcommand{\qflower}{\texttt{q}^{\text{L}}}	
\newcommand{\qfupperhat}{\widehat{\texttt{q}}^{\text{U}}}	
\newcommand{\qflowerhat}{\widehat{\texttt{q}}^{\text{L}}}	
\newcommand{\qfupperhatsafe}{\widehat{\texttt{q}}^{\mathrm{U}\star}}	
\newcommand{\qflowerhatsafe}{\widehat{\texttt{q}}^{\mathrm{L}\star}}	
\newcommand{\traceH}[1]{\trace^\mathrm{H}_{#1}}	
\DeclareMathOperator{\spec}{\Lambda} 		

\newcommand{\rev}[1]{#1}		

\numberwithin{equation}{section}

\graphicspath{{./figures/}}

\begin{document}

\title[Estimation of spectral gaps]{Estimation of spectral gaps for sparse symmetric matrices}
\author{Michele Benzi}
\address{Scuola Normale Superiore, Piazza dei Cavalieri, 7, 56126 Pisa, Italy}
\email{michele.benzi@sns.it}

\author{Michele Rinelli}
\address{Department of Computer Science, KU Leuven, Celestijnenlaan 200A, 3001 Leuven, Belgium}
\email{michele.rinelli@kuleuven.be}

\author{Igor Simunec}
\address{Institute of Mathematics, École Polytechnique Fédérale de Lausanne, 1015 Lausanne, Switzerland}
\email{igor.simunec@epfl.ch}

\subjclass[2010]{65F15, 65F60}

\keywords{spectral gap, spectral projector, trace estimation, Lanczos algorithm}

\begin{abstract}
	In this paper we propose and analyze an algorithm for identifying spectral gaps of a real symmetric matrix $A$ by simultaneously approximating the traces of spectral projectors associated with multiple different spectral slices. 
	Our method utilizes Hutchinson's stochastic trace estimator together with the Lanczos algorithm to approximate quadratic forms involving spectral projectors. 
	Instead of focusing on determining the gap between two particular consecutive eigenvalues of $A$, we aim to find all gaps that are wider than a specified threshold. 
	By examining the problem from this perspective, and thoroughly analyzing both the Hutchinson and the Lanczos components of the algorithm, we obtain error bounds that allow us to determine the number of Hutchinson's sample vectors and Lanczos iterations needed to ensure the detection of all gaps above the target width with high probability. 
	In particular, we conclude that the most efficient strategy is to always use a single random sample vector for Hutchinson's estimator and concentrate all computational effort in the Lanczos algorithm.
	Our numerical experiments demonstrate the efficiency and reliability of this approach.
\end{abstract}

\maketitle

\section{Introduction}
\label{sec:introduction}

We consider the problem of locating gaps in the spectrum of an $n \times n$ real symmetric matrix~$A$, i.e., finding intervals that do not contain any eigenvalue of~$A$.
A variant of this problem that is more commonly encountered in the literature is the following: given an integer $k$ with $1 \le k < n$, find a real number $\mu$ such that exactly $k$ of the eigenvalues of $A$ (counted with their multiplicities) are strictly below~$\mu$.
This problem can be equivalently restated as finding $\mu$ such that the spectral projector~$P_{\mu}$ onto the subspace spanned by the eigenvectors of $A$ associated with eigenvalues strictly smaller than~$\mu$ satisfies~$\rank(P_{\mu}) = k$; since for a projector the rank is equal to the trace, we can also write this requirement as~$\trace(P_\mu) = k$.
If we denote the eigenvalues of $A$ by $\lambda_i$ and we label them in non-decreasing order, this problem has a solution only if $\lambda_k < \lambda_{k+1}$ (non-degeneracy assumption).
In this case, any $\mu$ lying in the open interval $(\lambda_k, \lambda_{k+1})$ is a solution.
Instead of looking for the gap for a specific value of $k$, the approach that we present here aims to find all gaps in the spectrum with width above a certain threshold, and in addition we obtain an estimate of the number of eigenvalues below each gap.
In practical cases where it is of interest to find $\mu$ in the gap between $\lambda_k$ and $\lambda_{k+1}$, the gap $(\lambda_k, \lambda_{k+1})$ is often relatively large, so we expect to be able to find it by looking for all gaps above a certain width.

\rev{
There are several applications where it is required to locate gaps in the spectrum of a given symmetric matrix.  In Kohn-Sham Density Functional Theory for electronic structure computations \cite{LLY19}, at each step of a self-consistent field iteration it is necessary to compute either 
the eigenvectors of a symmetric matrix (the linearized discrete Kohn-Sham Hamiltonian) associated with eigenvalues below the {\it Fermi level} (or {\it chemical potential}, usually denoted by $\mu$), corresponding to the so-called {\it occupied states} of the system, or equivalently the corresponding spectral projector $P_{\mu}$. In the case of insulators at zero electronic temperature, the  linearized Hamiltonians exhibit a gap separating the first~$k$ eigenvalues from the rest of the spectrum, where $k$ is the number of electrons, and locating this gap is equivalent to estimating the Fermi level $\mu$.
 In particular, in linear scaling electronic structure computations the spectral projector is approximated either by polynomials or by rational functions that act as filters, 
 i.e., approximate the step function~$h_\mu(\lambda)$ that takes the value 1 for $\lambda < \mu$ and $0$ for $\lambda > \mu$ (see, e.g., \cite{BenziProjector13}). 
Clearly, in order to use this approach it is first necessary to estimate the Fermi level $\mu$, which is a nontrivial task.

Another potential area of application arises in solid state physics, where
the location of gaps in the energy spectrum of quantum-mechanical systems described by Schr\"odinger operators with periodic potentials
is an important problem; see, e.g., \cite{KM18}, as well as \cite{RS78}.

An additional possible area of application is {\em spectral clustering}, which uses the eigenvectors of the graph Laplacian to partition a graph, see \cite{AEGL21} and references therein. 
The stability of the partitioning under small perturbations of the weights is governed by the gaps in the spectrum of the graph Laplacian; in particular, relatively large gaps correspond to stable partitions. Hence, identifying sufficiently large gaps in the spectrum of the Laplacian matrix is an important preliminary step for any spectral clustering algorithm.
	
	Finally, as a further application we mention estimating the stability of invariant subspaces under perturbations.  It is well known in numerical linear algebra
	that individual eigenvectors of symmetric matrices corresponding to a cluster of tightly spaced eigenvalues will be highly ill-conditioned.  On the other hand, if the 
	cluster is well-separated from the rest of the spectrum, the invariant subspace spanned by those eigenvectors is well-conditioned and can be determined
	with high accuracy; furthermore, iterative methods designed to approximate the invariant subspace (or, equivalently, the orthogonal projector onto it)
	will typically converge rapidly. Hence, being able to efficiently identify gaps in the spectrum of a symmetric matrix $A$ provides a way to estimate the conditioning
	of invariant subspaces.  Moreover, when $A$ is banded, or sparse, the invariant space is ``localized", an important property that can reduce computational
	costs considerably; see, e.g., \cite{VomelParlett11} and also  \cite[Sec.~11.4]{BenziProjector13} and \cite{BenziRinelli22}.
}

In the literature, a few approaches have been proposed to estimate values of $\mu$ lying within the target gap.
It is well known that for any real symmetric matrix $A-\mu I$ there exists a unit lower triangular matrix $L$ and a block diagonal matrix $D$ with $1\times 1$ and $2\times 2$ blocks such that $A-\mu I = LDL^T$; see for example \cite[Chapter~1.3.4]{Bjorck}.
Since $A$ and $D$ are congruent, they have the same inertia, i.e., the same number of positive, negative, and possibly zero eigenvalues (the latter can occur only when $\mu$ happens to be an eigenvalue of $A$, an extremely unlikely occurrence in practice).
Hence, if $A$ is not too large, an $LDL^T$ factorization of $A-\mu I$ can be used to count the number of eigenvalues smaller than $\mu$. If this number is less than~$k$ (larger than $k$) then $\mu$ is increased (respectively, decreased) and the procedure is repeated until the correct value is found.
When $A$ is sparse, one can make use of symmetric row and column permutations to reduce the fill-in in $L$ (and therefore the arithmetic and storage costs) while preserving sparsity, see \cite[Chapter~7]{ST}.
A major drawback of this approach is its cost, and especially the fact that having computed the $LDL^T$ factorization of $A-\mu I$ for a certain~$\mu$ is of no help in computing the factorization for a different value of $\mu$, hence using matrix factorizations in a trial-and-error fashion can be prohibitively expensive, even assuming that a factorization can be computed at all.

Different techniques have been used in the context of linear scaling methods.
In these methods, an initial guess $\mu_0\approx \mu$ is used in constructing a polynomial approximation to the projector $P_{\mu_0}$ by an iterative process known as {\it purification}, essentially a Newton-type method.
The value of $\mu$ is adjusted in the course of the iterative process until the prescribed value $k$ of the trace is obtained; hence, the determination of $\mu$ and the approximation of the projector $P_\mu$ are interweaved.
The main cost in this approach is the need to repeatedly solve $n\times n$ linear systems with variable coefficient matrix at each iteration.
We refer to \cite{NTC03,Nik11} for details on these adaptive procedures.

Yet another approach is the one that combines stochastic trace estimators with the Lanczos algorithm to compute an approximation of $\trace(P_\mu)$. An advantage of this approach is that we can exploit the shift-invariance property of Krylov subspaces (see, e.g., \cite{FM99}) to efficiently approximate the trace of the spectral projectors $P_\mu$ associated with many different values of $\mu$ at the same time.
The combination of Hutchinson's estimator and the Lanczos algorithm has also appeared in literature on related problems that exploit the trace of spectral projectors, such as the estimation of the number of eigenvalues of a matrix~$A$ contained in a given interval $[a, b]$ or the estimation of spectral densities, see, for instance,~\cite{DPS16, LSY16, BKM22, CTU21}. In the context of spectral density approximation, a similar approach is the kernel polynomial method, which was introduced in \cite{SilverRoder94, Wang94} and uses a polynomial expansion based on Chebyshev polynomials. 

Here we focus on the approach that combines Hutchinson's trace estimator and the Lanczos algorithm to estimate the trace of the spectral projectors $P_\mu$. 
When using this kind of approach, it is essential to have a criterion to choose the parameters of the algorithm, namely the number of random vectors $\vec x_i$ for Hutchinson's estimator and the number of Lanczos iterations, in order to minimize the computational cost and guarantee the correctness of the results. However, although there exist theoretical results on the convergence speed of this approach \rev{in the literature, they are mainly focused on the global approximation of the spectral density and they are not satisfactory when applied to gap detection; in particular, they are not useful to determine} how the parameters should be selected.
The main purpose of this work is to answer this question for the problem of finding gaps in the spectrum of a matrix. By changing the point of view from finding $\mu$ such that the number of eigenvalues below $\mu$ is exactly~$k$ to finding all the gaps in the spectrum with width above a certain threshold, we are able to provide a rigorous analysis that allows us to choose the input parameters to ensure that all the target gaps are found, and at the same time minimize the computational cost. Our analysis will determine that for this problem the most efficient choice is to use a single quadratic form $\vec x^T P_\mu \vec x$ to approximate~$\trace(P_\mu)$, but in order to reach this conclusion we first have to analyze the more general algorithm that uses Hutchinson's stochastic trace estimator.
This analysis will result in an algorithm that simultaneously computes upper and lower bounds to $\vec x^T P_\mu \vec x$ for all considered values of $\mu$, and exploits them to find intervals that contain no eigenvalues of~$A$ with high probability.

The rest of the paper is organized as follows. 
In \cref{sec:algorithm-overview} we describe a basic version of the algorithm that we are going to use to detect the gaps in the spectrum of~$A$. 
Each component of the algorithm is then analyzed in detail in \cref{sec:algorithm-analysis}, with a focus on bounding the error. 
Then, in \cref{sec:final-algorithm} we put together the analysis to determine the best choice of input parameters and we obtain an efficient and robust algorithm that can give guarantees on its output. 
In \cref{sec:numerical-experiments} we conduct some numerical experiments to evaluate the performance and reliability of our approach.
Finally, \cref{sec:conclusions} contains some concluding remarks.

\section{Algorithm overview}
\label{sec:algorithm-overview}

We start by introducing some notation to establish the framework for our algorithm.
Given $\mu \in \R$, let  $h_\mu: \R \to \R$ denote the Heaviside function
\begin{equation*}
	h_\mu(x) = \begin{cases}
		1   & \text{for $x < \mu$,} \\
		1/2 & \text{for $x = \mu$,} \\
		0   & \text{for $x > \mu$.}
	\end{cases}
\end{equation*}
The step function $h_\mu$ is closely related to the sign function, indeed we have $h_\mu(x) = -\frac{1}{2} \sign(x - \mu) + \frac{1}{2}$.
We denote by $\lambda_i$ the eigenvalues of $A$ in non-decreasing order, counted according to their multiplicities, and by $\{\vec u_i\}_{i = 1}^n \subset \R^n$ an orthonormal basis of eigenvectors, where $\vec u_i$ is an eigenvector associated with $\lambda_i$.
Assuming that $\mu$ does not coincide with any of the eigenvalues of the matrix~$A$, we can define the spectral projector
\begin{equation*}
	P_\mu := h_\mu(A) = \sum_{i \, : \, \lambda_i < \mu} \vec u_i \vec u_i^T.
\end{equation*}
\rev{Note that although we can define $h_\mu(A)$ for all $\mu$, since $h_\mu(\mu) = 1/2$ this definition only yields a spectral projector if $\mu \ne \lambda_i$ for all $i = 1, \dots, n$, so in the following we are always going to assume that this condition holds.} 
Since the eigenvalues of $P_\mu$ are all either $0$ or $1$, we have
\begin{equation*}
	\trace(P_\mu) = \rank(P_\mu) = \card\{ i \, : \, \lambda_i < \mu \} =: \eigcount(\mu),
\end{equation*}
where $\eigcount(\mu)$ denotes the number of eigenvalues of $A$ that are strictly smaller than~$\mu$.
The problem of finding gaps in the spectrum of $A$ can be formulated as finding intervals $[a_j, b_j]$ such that $\trace(P_\mu)$ is constant for $\mu \in [a_j, b_j]$. We are going to determine gaps in the spectrum of $A$ by approximating~$\trace(P_\mu)$ for several different values of $\mu$, using Hutchinson's trace estimator combined with the Lanczos method for the approximation of quadratic forms.
We first focus on the approximation of~$\trace(P_\mu)$ for a fixed value of $\mu$.

\subsection{Hutchinson's trace estimator}
\label{subsec:hutchinson-trace-estimator}

Stochastic trace estimators approximate the trace of a matrix $B$ that is accessible via matrix-vector products by making use of the fact that, for any random vector~$\vec x$ such that $\expec [\vec x \vec x^T] = I$, we have $\expec [\vec x^T B \vec x] = \trace(B)$, where $\expec$ denotes the expected value. 
Hutchinson's trace estimator~\cite{Hutchinson89} is a simple stochastic estimator that generates~$s$ vectors $\vec x_1, \dots, \vec x_s$ with i.i.d.~random entries, that are usually chosen as either Gaussian $\mathcal{N}(0,1)$ or Rademacher (uniform $\pm 1$), and approximates~$\trace(B)$ with
\begin{equation}
	\label{eqn:projector-hutchinson-trace-approximation}
	\traceH{s}(B) := \frac{1}{s} \sum_{i = 1}^s \vec x_i^T B \vec x_i.
\end{equation}
In order to attain an absolute accuracy~$\varepsilon$ on $\trace(B)$, Hutchinson's trace estimator requires $O(\norm{B}_F^2/\varepsilon^2)$ sample vectors \cite{CortinovisKressner22,RKAscher15}. 
More precisely, the error from the Hutchinson's trace estimator when $B$ is symmetric can be bounded by using the following result from \cite{CortinovisKressner22}.

\begin{proposition}[{\hspace{1sp}\cite[Theorem~1]{CortinovisKressner22}}]
	\label{prop:hutchinson-tail-bound--gaussian}
	Let $B$ be a nonzero symmetric matrix, and let $\traceH{s}(B)$ denote the Hutchinson trace approximation obtained using $s$ i.i.d. Gaussian $\mathcal{N}(0,1)$ random vectors.
	If
	\begin{equation*}
		s \ge \frac{4}{\varepsilon^2} \big(\norm{B}_F^2 + \varepsilon \norm{B}_2\big) \log(2/\delta),
	\end{equation*}
	then
	\begin{equation*}
		\prob \big( \abs{\trace(B) - \trace_s^H(B)} \ge \varepsilon \big) \le \delta.
	\end{equation*}
\end{proposition}
If we use vectors with Rademacher instead of Gaussian entries, the convergence rate of Hutchinson's estimator depends on $\norm{B - \diag(B)}_F$ and $\norm{B - \diag(B)}_2$ instead of $\norm{B}_F$ and $\norm{B}_2$, so in certain situations it can be faster than with Gaussian vectors, for instance when $B$ is diagonally dominant; see \cite[Corollary~1]{CortinovisKressner22} for a precise statement.

The $O(1/\varepsilon^2)$ dependence on $\varepsilon$ of the number of sample vectors $s$ makes it very expensive to obtain highly accurate approximations with stochastic trace estimators. However, this is not an issue for a setting such as ours, where an absolute accuracy $\varepsilon = \frac{1}{2}$ is enough to exactly determine $\trace(P_\mu)$, which is an integer. In contrast, the dependence on $\norm{P_\mu}_F^2$ will make it quite difficult to accurately estimate the number of eigenvalues below~$\mu$, since $\norm{P_\mu}_F^2 = \eigcount(\mu)$ can be quite large.
Note that when $\mu$ is close to the right endpoint of the spectrum, one can consider $I-P_\mu$ instead of $P_\mu$ in order to reduce its norm, since $\trace(I-P_\mu)=n-\eigcount(\mu)$.
However, this does not handle the case where $\mu$ is near the center of the spectrum. 
To overcome this issue, we will have to slightly change our perspective in order to develop an efficient algorithm. This topic will be discussed in greater detail in \cref{subsec:algorithm-analysis--hutchinson}.

\begin{remark}
	\label{rem:other-trace-estimators}
	Since $\rank(P_\mu) = \eigcount(\mu)$ and all the eigenvalues of $P_\mu$ are either~$0$ or~$1$, for this setting it usually does not make much sense to use the improved stochastic trace estimators such as Hutch++ \cite{MMMW21-Hutch++,PerssonCortinovisKressner21} or XTrace \cite{XTrace}, which also employ a low-rank matrix approximation. Indeed, there is no gradual decay in the eigenvalues of $P_\mu$ and hence there is no advantage in computing a low-rank approximation of $P_\mu$, unless $\eigcount(\mu)$ is very small compared to the matrix size.
\end{remark}

\begin{remark}
	\label{rem:probing-methods}
	When the matrix $A$ is sparse, we could alternatively approximate the trace of the matrix function $P_\mu = h_\mu(A)$ with a probing method (see, for instance, \cite{FrommerProbing21,BRS23,frommer2023stochasticprobing}).
	The probing approach has a faster convergence when the function $h_\mu(A)$ is well-approximated by polynomials on the spectrum of~$A$~\cite[Theorem~4.4]{FrommerProbing21}, i.e., it converges similarly to the Lanczos method for the computation of $h_\mu(A) \vec x$. We will see in \cref{subsec:algorithm-analysis--lanczos} that this convergence rate is heavily dependent on the value of $\mu$, so we prefer to use a stochastic trace estimator in order to eliminate this dependence.
\end{remark}

\subsection{Lanczos approximation of quadratic forms}
\label{subsec:lanczos-approximation-of-quadratic-forms}

Within Hutchinson's estimator we have to compute the quadratic forms $\vec x_i^T P_\mu \vec x_i$ for $i = 1, \dots, s$, which can be done efficiently by using the Lanczos method. 
For simplicity, consider a single quadratic form~$\vec x^T h_\mu(A) \vec x$.
Let us denote by $\kryl_m(A, \vec x) = \vspan\{ \vec x, A \vec x, \dots, A^{m-1} \vec x \}$ the Krylov subspace of dimension $m$ associated to $A$ and $\vec x$, and let $V_m \in \R^{n \times m}$ be an orthonormal basis of $\kryl_m(A, \vec x)$ constructed using the Lanczos algorithm~\cite[Algorithm~6.15]{Saad03}. 
The Lanczos algorithm also constructs a tridiagonal matrix $\underline{T_m} \in \R^{(m+1) \times m}$ that satisfies the Arnoldi relation
\begin{equation*}
	A V_m = V_{m+1} \underline{T}_m.
\end{equation*}
The quadratic form $\vec x^T P_\mu \vec x$ can then be approximated with
\begin{equation}
	\label{eqn:quadform-lanczos-approximation}
	\quadform_m := \norm{\vec x}_2^2 \vec e_1^T h_\mu(T_m) \vec e_1,
\end{equation}
where $T_m = V_m^T A V_m$ is the leading principal $m \times m$ block of $\underline{T}_m$.
The convergence rate of the approximation in \cref{eqn:quadform-lanczos-approximation} depends on how well the function $f$ can be approximated with polynomials over the spectrum of $A$; see~\cref{subsec:algorithm-analysis--lanczos}. The Lanczos approximation of $\vec x^T P_\mu \vec x$ can also be interpreted and analyzed in terms of Gaussian quadrature rules~\cite{GolubMeurant10}.

\subsection{Simultaneous computation for different $\mu$}
\label{subsec:simultaneous-computation-for-different-mu}

The approach described above can be easily modified in order to compute approximations to $\trace(P_\mu)$ simultaneously for several $\mu$, for a cost that is only slightly higher than the cost of the approximation for a single $\mu$.
The key observation is that for a fixed matrix $A$ and vector $\vec x$, we can extract approximations to the quadratic forms $\vec x^T f_j(A) \vec x$ for several different functions $f_j$ from a single Krylov subspace~$\kryl_m(A, \vec x)$.
Since the cost of computing the approximation \cref{eqn:quadform-lanczos-approximation} is dominated by the cost of constructing the orthonormal basis~$V_m$ via the Lanczos algorithm, the additional operations required to compute approximations for different functions represent only a minor part of the overall cost.
More precisely, let us consider $N_f$ different functions $f_j$, $j = 1, \dots, N_f$. For a fixed Krylov subspace dimension~$m$, we can approximate each quadratic form $\vec x^T f_j(A) \vec x$ with
\begin{equation*}
	\quadform_{m, j} := \norm{\vec x}_2^2 \vec e_1^T f_j(T_m) \vec e_1, \qquad j = 1, \dots, N_f.
\end{equation*}
The quantities $\quadform_{m, j}$ can be computed efficiently in the following way. Obtain first an eigenvalue decomposition $T_m = U_m D_m U_m^T$ and compute $f_j(D_m)$ for $j = 1, \dots, N_f$ by applying the functions $f_j$ to the diagonal entries of $D_m$. The Lanczos approximations to the quadratic forms can now be computed with the identity
\begin{equation*}
	\quadform_{m, j} = \vec w_m^T f_j(D_m) \vec w_m, \qquad j = 1, \dots, N_f,
\end{equation*}
where $\vec w_m := \norm{\vec x}_2 U_m^T \vec e_1$ is the same for all $j$ and needs to be computed only once.
If we assume that the small matrix function $f(T_m)$ is computed via an eigenvalue decomposition also in the case of a single function $f$, the only additional operations that we have to do when we have $N_f$ different functions are the computation of the diagonal matrix functions~$f_j(D_m)$ and the computation of the final approximations $\quadform_{m, j} = \vec w_m^T f_j(D_m) \vec w_m$ for all $j$, for a total cost of $O(m N_f)$. This additional cost is negligible compared to the $O(m \, \nnz(A))$ cost for the matrix-vector products with $A$ in the construction of the Krylov subspace $\kryl_m(A, \vec x)$, unless the number of functions~$N_f$ is extremely large.

In total, the cost of computing approximations to $\vec x^T f_j(A) \vec x$ for $j = 1, \dots, N_f$ with the Lanczos algorithm adds up to $O(m \, \nnz(A) + nm + m^2 + m N_f)$, where the $O(n m)$ term is the cost of orthogonalization in the Lanczos algorithm using the three-term recurrence; if full reorthogonalization is used instead to increase stability, the orthogonalization cost would increase to $O(n m^2)$. The $O(m^2)$ term corresponds to the computation of eigenvalues and eigenvectors\footnote{ \rev{The $O(m^2)$ cost for computing eigenvectors of a tridiagonal matrix assumes that the MRRR algorithm \cite{DhillonParlett03,DhillonParlett04} is used, which is now the state of the art in SCALAPACK and other software packages. We note that this algorithm might fail in some extreme cases; however, remedies are known \cite{DhillonParlett05}.}} of the tridiagonal matrix $T_m$, and the~$O(m N_f)$ term corresponds to the computation of $\quadform_{m,j}$ for all $j$.
Note that the convergence rate of the approximation of the quadratic form $\vec x^T f_j(A) \vec x$ can vary significantly depending on the function~$f_j$, so this approach may construct approximations with considerably different errors for different functions, since we are using the same number of iterations $m$ for all functions. We will focus on analyzing this aspect in \cref{subsec:algorithm-analysis--lanczos}.

\begin{remark}
	\label{rem:rational-krylov-subspaces}
	A similar approach could be used to approximate quadratic forms with several different functions using rational Krylov methods~\cite{Guettel13}, which depend on a certain sequence of poles~$\vec \xi = [\xi_1, \dots, \xi_{m-1}]$ and employ as an approximation space the rational Krylov subspace $\rat_m(A, \vec x, \vec \xi) = q_{m-1}(A)^{-1} \kryl_m(A, \vec b)$, with $q_{m-1}(z) := \prod_{j = 1}^{m-1}(z - \xi_j)$.
	However, the convergence of a rational Krylov method is heavily dependent on the poles $\vec \xi$, which should be chosen depending on the function. In particular, in the case of $f_j = h_{\mu_j}$ for different parameters $\mu_j$, a sequence of poles that produces an accurate approximation for a certain $\mu_j$ will give poor approximations for other parameters, especially if they are far from $\mu_j$. For this reason, the Lanczos method is better suited for this scenario. A rational Krylov subspace method could be useful to compute a more accurate approximation for the quadratic form $\vec x^T h_\mu(A) \vec x$ associated with a specific parameter $\mu$, especially in situations in which the Lanczos method converges very slowly.
\end{remark}

Returning now to the approximation of $\trace(P_\mu)$ with Hutchinson's estimator, let us consider $N_f$ different parameters $\{\mu_j\}_{j=1}^{N_f} \subset \R$ and define the functions $f_j = h_{\mu_j}$. Provided that we use the same random sample vectors $\{\vec x_i\}_{i=1}^s$ for the approximation of $\trace(P_{\mu_j})$ via Hutchinson's estimator for all $j$, we can use the same Krylov subspace $\kryl_m(A, \vec x_i)$ to extract approximations to $\vec x_i^T P_{\mu_j} \vec x_i$ for all $\mu_j$ at the same time, with essentially the same cost of approximating a quadratic form for a single parameter $\mu$.
This feature of this approach represents a considerable computational advantage over other methods, such as those based on the $LDL^T$ factorization, which are only able to test a single candidate $\mu$ at a time.

A high-level algorithmic description of the procedure outlined in this section is given in \cref{algorithm:eigenvalue-gap-finder--basic}. This algorithm returns an approximation to $\trace(P_{\mu_j})$ for all $j = 1, \dots, N_f$.
The level of accuracy of these approximations, as well as whether they can be exploited to detect gaps in the spectrum of $A$, is going to be thoroughly investigated in the sections that follow.
We are also going to identify criteria for selecting the input parameters of the method, namely the number of Hutchinson sample vectors $s$ and Lanczos iterations $m$.
The product of this analysis will be an algorithm that, given as input a failure probability $\delta$ and a relative gap width $\theta$, aims to find all the gaps in the spectrum whose relative width is larger than~$\theta$; each one of these gaps is found with probability at least~$1 - \delta$.
In addition, all the gaps found by the algorithm are certified, in the sense that the probability of finding an eigenvalue of~$A$ within a gap is smaller than $\delta$.

\rev{
\begin{remark}
	\label{rem:literature-comparison}
	We mention that \cref{algorithm:eigenvalue-gap-finder--basic} has been extensively used in the literature for related problems, such as spectral density estimation and estimation of eigenvalue counts in an interval \cite{CTU21,BKM22,DPS16}. The main novelty of our work is in the theoretical analysis contained in the following sections, which provides guarantees for the detection of spectral gaps, and allows us to select the parameters $s$ and $m$ in order to minimize the computational cost. We are going to provide a more detailed comparison of our contributions with those found in the literature in \cref{subsec:literature-comparison}.    
\end{remark}
}

\begin{algorithm}
	\caption{Eigenvalue gap finder -- Prototype}
	\label{algorithm:eigenvalue-gap-finder--basic}
	\begin{algorithmic}[1]
		\Require Symmetric matrix $A \in \R^{n \times n}$, number of Hutchinson samples $s$, Krylov subspace dimension $m$, vector of test parameters $\vec \mu = [\mu_1, \dots, \mu_{N_f}]$
		\Ensure $\texttt{tr}_j \approx \trace(P_{\mu_j})$ for $j = 1, \dots, N_f$
		\For {$i = 1, \dots, s$}
		\State Draw a vector $\vec x_i \in \R^n$ with random i.i.d.~Gaussian $\mathcal{N}(0, 1)$ entries
		\State Construct an orthonormal basis $V_m \in \R^{n \times m}$ of the Krylov space $\kryl_m(A, \vec x_i)$ with the Lanczos algorithm and compute the associated tridiagonal matrix $\underline{T}_m \in \R^{(m+1) \times m}$
		\State Compute the eigenvalue decomposition $T_m = U_m D_m U_m^T$ and $\vec w_m = \norm{\vec x_i}_2 U_m^T \vec e_1$
		\For {$j = 1, \dots, N_f$}
		\State Compute $\quadform_{m,j}^{(i)} = \vec w_m^T h_{\mu_j}(D_m)\vec w_m$
		\EndFor
		\EndFor
		\For { $j = 1, \dots, N_f$ }
		\State Compute the trace approximation $\texttt{tr}_j = s^{-1} \sum_{i = 1}^s \quadform_{m, j}^{(i)}$
		\EndFor
	\end{algorithmic}
\end{algorithm}

\section{Algorithm analysis}
\label{sec:algorithm-analysis}

In this section we examine some general properties of the trace approximations computed by \cref{algorithm:eigenvalue-gap-finder--basic}, and we analyze in detail the errors that occur in the Hutchinson and Lanczos parts of the approximation.

\subsection{General properties}
\label{subsec:general-properties}

A simple but essential observation is that the approximations to~$\trace(P_\mu)$ computed by \cref{algorithm:eigenvalue-gap-finder--basic} are increasing in $\mu$.
Let us start by looking at the Hutchinson approximation. Denoting by $(\lambda_j, \vec u_j)_{j = 1}^n$ the normalized eigenpairs of $A$, so that $A = \sum_{j = 1}^n \lambda_j \vec u_j \vec u_j^T$, we can write the trace approximation obtained with Hutchinson's estimator as
\begin{equation*}
	\traceH{s}(P_\mu) = \frac{1}{s} \sum_{i = 1}^s \vec x_i^T P_\mu \vec x_i = \frac{1}{s} \sum_{i = 1}^s \vec x_i^T \Big( \sum_{j \,:\, \lambda_j < \mu} \vec u_j \vec u_j^T \Big) \vec x_i = \frac{1}{s} \sum_{i = 1}^s \sum_{j \,:\, \lambda_j < \mu} \abs{\vec u_j^T \vec x_i}^2.
\end{equation*}
This shows that the Hutchinson trace approximation is monotonic in $\mu$. In addition, the above identity also highlights the fact that Hutchinson's approximation of $\trace(P_\mu)$ is a piecewise constant function, with jumps when $\mu$ coincides with an eigenvalue of $A$. The height of the jump at $\lambda_k$ is given by~$\frac{1}{s} \sum_{i = 1}^s \abs{\vec u_k^T \vec x_i}^2$.

Let us now consider the Lanczos approximation of the quadratic form $\vec x_i^T P_\mu \vec x_i$. Letting~$T_m^{(i)}$ be the projection of $A$ onto $\kryl_m(A, \vec x_i)$ computed by the Lanczos algorithm, and denoting by $(\lambda_j^{(i)}, \vec u_j^{(i)})_{j = 1}^m$ the normalized eigenpairs of $T_m^{(i)}$, we have
\begin{equation*}
	\vec x_i^T P_\mu \vec x_i \approx \quadform_{m, j}^{(i)} := \norm{\vec x_i}_2^2 \vec e_1^T h_\mu(T_m^{(i)}) \vec e_1 = \norm{\vec x_i}_2^2 \sum_{j \,:\, \lambda_j^{(i)} < \mu} \abs{\vec e_1^T \vec u_j^{(i)}}^2.
\end{equation*}
Again, this is an increasing function of $\mu$, with jumps whenever $\mu$ coincides with an eigenvalue of~$T_m^{(i)}$. It follows that the approximation of $\trace(P_\mu)$ computed by \cref{algorithm:eigenvalue-gap-finder--basic} using Hutchinson's estimator and the Lanczos method is an increasing function of $\mu$, with jumps when $\mu$ is equal to an eigenvalue of~$T_m^{(i)}$ for some $i$.

\begin{example}
	\label{example:hutchinson-projector-traceest}
	To illustrate the behavior of the approximation of $\trace(P_\mu)$ computed by \cref{algorithm:eigenvalue-gap-finder--basic} on a simple example, we consider a symmetric matrix $A$ of size $n = 600$, with eigenvalues contained in the interval $[0, 60]$ and three gaps given by the intervals $[20, 21]$, $[30, 32]$ and $[40, 44]$. In \cref{fig:example-hutchinson-projector-traceest} we compare the approximation from \cref{algorithm:eigenvalue-gap-finder--basic} with $m = 50$ Lanczos iterations, Hutchinson's estimator~$\traceH{s}(P_\mu)$ with exact quadratic forms, and the exact trace $\trace(P_\mu)$, which coincides with $\eigcount(\mu)$, for several different values of $\mu$. The number of random vectors is either $s = 1$ or $s = 10$. 
	\cref{fig:example-hutchinson-projector-traceest} confirms the behavior described above. In particular, we see that Hutchinson's estimator with exact quadratic forms has a jump whenever $\mu$ coincides with an eigenvalue of $A$, so it correctly detects all the gaps in the spectrum, although the value of the approximation of $\trace(P_\mu)$ can be quite off, especially with $s = 1$.  
	On the other hand, Hutchinson's estimator combined with the Lanczos algorithm for computing quadratic forms, which is the approximation that we are actually able to compute, has jumps when $\mu$ corresponds to an eigenvalue of one of the projected matrices $T_m^{(i)}$. We see that the Lanczos approximation is able to detect the largest gap quite well, but the other two gaps do not stand out from the other steps in the staircase-like curve. With $s = 10$, the Lanczos curve looks like a staircase with somewhat jagged steps, because the projected matrices $T_m^{(i)}$ have different eigenvalues for different random vectors $\vec x_i$, and the curve is obtained by averaging the approximations~$\quadform_{m, j}^{(i)}$  associated with each $T_m^{(i)}$. \rev{Note that the curves in \cref{fig:example-hutchinson-projector-traceest} are approximations of the (unnormalized) spectral density of $A$, which coincide with the approximations shown, for instance, in \cite[Figure~1]{CTU21}.}
\end{example}

\begin{figure}[t]
	\centering
	\makebox[\linewidth][c]{
		\begin{subfigure}[t]{.50\textwidth}
			\includegraphics[width=\textwidth]{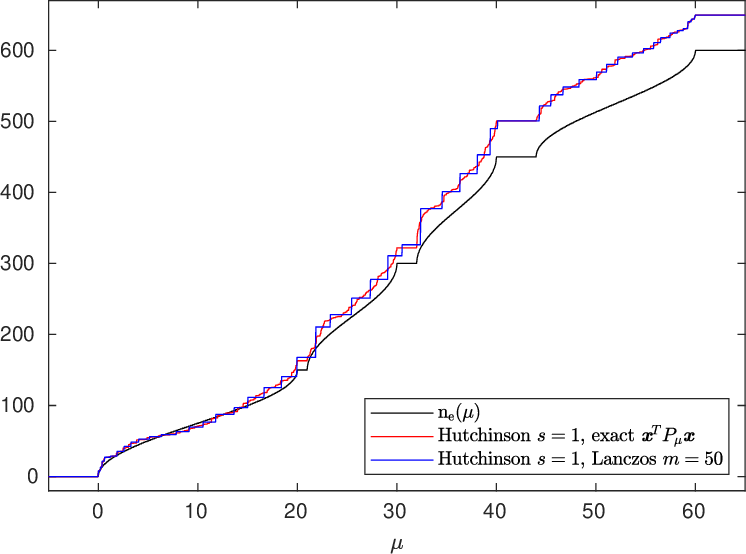}
		\end{subfigure}
		\begin{subfigure}[t]{.50\textwidth}
			\includegraphics[width=\textwidth]{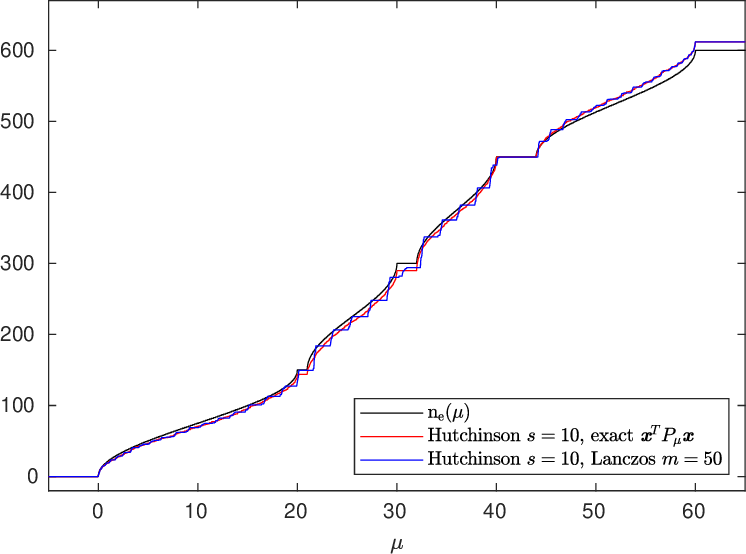}
		\end{subfigure}
		}
	\caption[Approximation of $\trace(P_\mu)$ with Hutchinson's estimator for several $\mu$]{Approximations to $\trace(P_\mu)$ obtained with \cref{algorithm:eigenvalue-gap-finder--basic} compared with exact $\trace(P_\mu)$ and Hutchinson's approximation $\traceH{s}(P_\mu)$ with exact quadratic forms, for several different $\mu$. Left: $s = 1$, $m = 50$. Right: $s = 10$, $m = 50$.   
	\label{fig:example-hutchinson-projector-traceest}}
\end{figure}

\subsection{Trace estimator analysis}
\label{subsec:algorithm-analysis--hutchinson}

In this section we focus on analyzing the approximation of $\trace(P_\mu)$ using Hutchinson's estimator.
The error from the stochastic trace estimator can be bounded with \cref{prop:hutchinson-tail-bound--gaussian}.
Since all the eigenvalues of $P_\mu$ are either $0$ or $1$, we have $\norm{P_\mu}_2 = 1$ and $\norm{P_\mu}_F^2 = \eigcount(\mu)$.
Therefore, in order to have $\prob\big( \abs{\trace(P_\mu) - \trace_p^H(P_\mu)} \ge \varepsilon \big) \le \delta$ with Gaussian vectors, we would need to take
\begin{equation}
	\label{eqn:hutchinson-accuracy--gaussian-cost-requirement}
	s \ge \frac{4}{\varepsilon^2} \big( \eigcount(\mu) + \varepsilon \big) \log(2/\delta).
\end{equation}
Since the exact $\trace(P_\mu)$ only takes integer values, in order to compute it correctly it is sufficient to have a final absolute accuracy smaller than $\frac{1}{2}$ on the trace approximation; if we also account for the error in the Lanczos approximation, it is enough to take $\varepsilon = \frac{1}{4}$.
However, a key issue with this approach is the fact that the required number of sample vectors $s$ grows proportionally to the number~$\eigcount(\mu)$ of eigenvalues below the level $\mu$. This means that in a situation in which a gap is not close to the extrema of the spectrum, such as if $\eigcount(\mu) = n/c$ for some moderate constant $c > 1$, we would have to take~$s = O(n)$, which is too expensive to be practical.

\begin{remark}
	\label{rem:rademacher-vectors}
	In alternative to Gaussian random vectors, we can also use Rademacher vectors, which have random $\pm 1$ entries. As we mentioned in \cref{subsec:hutchinson-trace-estimator}, the convergence rate of Hutchinson's estimator with Rademacher random vectors depends on $\norm{B - \diag(B)}_F^2$ instead of $\norm{B}_F^2$, so it can be faster than with Gaussian vectors in certain scenarios, such as when the entries of $B$ decay away from the diagonal. Here we prefer to employ Gaussian vectors because of the rotational invariance of their distribution, which will prove useful in the following analysis.
\end{remark}

In order to circumvent the growth of the number of sample vectors with $\eigcount(\mu)$, we slightly change our point of view.
Instead of considering the approximation of $\trace(P_\mu)$, let us take two different values $\mu < \mu'$, and consider the difference $\trace(P_{\mu'}) - \trace(P_{\mu})$.
If we use the same sample vectors in Hutchinson's estimator for both $\trace(P_\mu)$ and $\trace(P_{\mu'})$, the difference $\traceH{s}(P_{\mu'}) - \traceH{s}(P_{\mu})$ coincides with the Hutchinson approximation of $\trace(P_{\mu'} - P_{\mu})$. Since we have $\norm{P_{\mu'} - P_{\mu}}_F^2 = \eigcount(\mu') - \eigcount(\mu)$, this means that Hutchinson's estimator can easily approximate the differences in the traces of projectors when $\mu$ and $\mu'$ are close enough that there are only a few eigenvalues of $A$ in the interval $[\mu, \mu']$, even if the exact values of~$\trace(P_\mu)$ and $\trace(P_{\mu'})$ are quite far from the approximations computed by the stochastic estimator. In particular, the heights of the jumps in $\traceH{s}(P_\mu)$ that occur when $\mu$ coincides with an eigenvalue of $A$ should be approximated quite accurately even with a small number of sample vectors.

Let us analyze this phenomenon in more detail. Recall that
\begin{equation*}
	\traceH{s}(P_\mu) = \frac{1}{s} \sum_{i = 1}^s \sum_{j \,:\, \lambda_j < \mu} \abs{\vec u_j^T \vec x_i}^2,
\end{equation*}
so if we take $\mu$ and $\mu'$ such that $\lambda_{k-1} < \mu < \lambda_k < \mu' < \lambda_{k+1}$ for some $k$, we have
\begin{equation*}
	\traceH{s}(P_{\mu'}) - \traceH{s}(P_{\mu}) = \frac{1}{s} \sum_{i = 1}^s \abs{\vec u_k^T \vec x_i}^2.
\end{equation*}
Since the random vectors $\vec x_i$ are independent of the eigenvector $\vec u_k$, we have that $\{ \vec u_k^T \vec x_i \}_{i = 1}^s$ are i.i.d.~$\mathcal{N}(0,1)$ random variables, due to the rotational invariance of Gaussian vectors. As a consequence, the random variable $y_s = \sum_{i = 1}^s \abs{\vec u_k^T \vec x_i}^2$ has a $\chi^2$ distribution with $s$ degrees of freedom, i.e., it is the sum of the squares of $s$ independent unit normal random variables.

We have $\expec[y_s] = s$, which means that the expected height of each jump in $\traceH{s}(P_\mu)$ is equal to~$1$.
Intuitively, the height of a jump in $\traceH{s}(P_\mu)$ is small when the vectors $\vec x_i$ are all approximately orthogonal to the eigenvector $\vec u_k$, an event that has a small chance of occurring.
Since small jumps in $\traceH{s}(P_\mu)$ are more difficult to detect when approximating the quadratic forms in Hutchinson's estimator via the Lanczos method, it will be useful to have an upper bound on the probability of small jumps.
The lemma that follows provides us with such a bound. In the following, we say that the jump in $\traceH{s}(P_\mu)$ at the $k$-th eigenvalue of $A$ is \emph{$\varepsilon$-small} if its height is smaller than $\varepsilon$, i.e., if $y_s \le s \varepsilon$.

\begin{lemma}
	\label[lemma]{lemma:small-jump-probability}
	Given $\delta \in (0, 1)$ and $\varepsilon \in (0,1)$, if \rev{the number of samples $s$ satisfies}
	\begin{equation*}
		s \ge \ceil{\frac{2 \log(1/\delta)}{\log(1/\varepsilon) + \varepsilon - 1}},
	\end{equation*}
	then we have
	\begin{equation*}
		\prob(y_s \le s \varepsilon) \le \delta,
	\end{equation*}
	i.e., the jump at $\lambda_k$ is $\varepsilon$-small with probability at most $\delta$.
\end{lemma}
\begin{proof}
	To prove this, we use a Chernoff bound on the lower tail of the probability distribution of $y_s$. 
	For any $\lambda > 0$, we have 
	\begin{equation}
		\label{eqn:proof-chernoff-bound-hutchinson}
		\prob(y_s \le s \varepsilon) = \prob(e^{-\lambda y_s} \ge e^{-\lambda s \varepsilon}) \le \frac{\expec[e^{-\lambda y_s}]}{e^{-\lambda s \varepsilon}},
	\end{equation}
	where the last inequality follows from Markov's inequality. Since $y_s$ is a $\chi^2$ random variable with $s$ degrees of freedom, we can write $y_s = X_1^2 + \dots + X_s^2$, where $\{X_j\}_{j = 1}^s$ are i.i.d.~$\mathcal{N}(0,1)$ random variables. 
	We have
	\begin{equation*}
		\expec[e^{-\lambda X_j^2}] = \frac{1}{\sqrt{2 \pi}} \int_{\R} e^{-\lambda z^2} e^{-\frac{1}{2} z^2}\de z = \frac{1}{\sqrt{2 \pi}} \int_{\R} \frac{1}{\sqrt{1+2\lambda}} e^{-\frac{1}{2} y^2}\de y = \frac{1}{\sqrt{1 + 2\lambda}},
	\end{equation*}
	where we used the change of variables $y = z \, \sqrt{1+2\lambda}$. It follows that
	\begin{equation*}
		\expec[e^{-\lambda y_s}] = \expec\bigg[\prod_{j = 1}^s e^{-\lambda X_j^2}\bigg] = \expec[e^{-\lambda X_1^2}]^s = (1+2\lambda)^{-s/2},
	\end{equation*}
	and by combining this identity with \cref{eqn:proof-chernoff-bound-hutchinson} we obtain
	\begin{equation}
		\label{proof-chernoff-bound-hutchinson-2}
		\prob(y_s \le s \varepsilon) \le e^{\lambda s \varepsilon} (1+2\lambda)^{-s/2} \qquad \forall \lambda > 0.
	\end{equation} 
	Taking $\lambda = \frac{1}{2 \varepsilon} - \frac{1}{2}$, we get
	\begin{equation*}
		\prob(y_s \le s \varepsilon) \le \big(\varepsilon e^{1 - \varepsilon}\big)^{s/2}.
	\end{equation*}
	With some simple algebraic manipulations, we conclude that $\prob(y_s \le s \varepsilon)\le \delta$ provided that we take
	\begin{equation*}
		s \ge \frac{2 \log(1/\delta)}{\log(1/\varepsilon) + \varepsilon - 1}.
	\end{equation*}
	\rev{This concludes the proof, since the number of samples $s$ is an integer.}
\end{proof}

Note that $\prob(y_s \le s \varepsilon)$ is the probability that a jump in $\traceH{s}(P_\mu)$ corresponding to a certain fixed eigenvalue of $A$ is $\varepsilon$-small, \rev{and we can use \cref{lemma:small-jump-probability} to obtain conditions on $s$ that bound the probability of having $\varepsilon$-small jumps associated with any set of eigenvalues of $A$}. \rev{This probability can be either bounded with a union bound argument, or computed directly by observing} that the heights of different jumps in $\traceH{s}(P_\mu)$ are independent, because the quantities $\{\vec u_k^T \vec x_i\}_{i, k}$ are all independent due to the rotational invariance of the $\vec x_i$ and the orthogonality of the $\vec u_k$.
The number of sample vectors $s$ required to have an $\varepsilon$-small jump probability lower than $\delta$ reduces as $\varepsilon$ decreases. At the same time, the number of Lanczos iterations required to approximate the quadratic forms with accuracy of order $\varepsilon$ (and thus manage to detect the jumps in $\traceH{s}(P_\mu))$ increases as~$\varepsilon$ becomes small. In \cref{subsec:algorithm-analysis--lanczos} we obtain some bounds on the Lanczos error, which we will later compare with the $\varepsilon$-small jump probability bound from \cref{lemma:small-jump-probability} in order to determine the best choice of parameters for the algorithm.

\subsection{Lanczos approximation analysis}
\label{subsec:algorithm-analysis--lanczos}

In this section we discuss the error in the Lanczos approximation of a quadratic form $\vec x^T h_\mu(A) \vec x$. Recall that we denote by
\begin{equation*}
	E_k(f, \mathcal{S})=\min_{\deg p \leq k} \max_{z\in \mathcal{S}} |f(z)-p(z)|
\end{equation*}
the best uniform norm error in the approximation of a continuous function $f$ on the compact set $\mathcal{S} \subset \R$ with polynomials of degree up to~$k$.
If $\quadform_m$ denotes the Lanczos approximation to~$\vec x^T f(A) \vec x$ from the Krylov subspace $\kryl_m(A, \vec x)$, we have
\begin{equation}
	\label{eqn:lanczos-matfun-approx--basic-bound}
	\abs{\vec x^T f(A) \vec x - \quadform_m} \le 2 \norm{\vec x}_2^2 E_{2m-1}(f, \spec(A) \cup \spec(T_m)),
\end{equation}
see, for instance, \cite[eq.~(1.2)]{CGMM21}.
Since the function $f = h_\mu$ is discontinuous at $\mu$, we have that $E_{2m-1}(h_\mu, [a, b])$ is always at least $\frac{1}{2}$ if $\mu \in [a, b]$, so we cannot derive meaningful error bounds if we simply consider an interval $[a, b]$ that contains the whole spectrum of $A$. On the other hand, we can obtain bounds on the polynomial approximation error of $h_\mu$ on two disjoint intervals, one to the left and one to the right of~$\mu$.
We start by stating a result for the polynomial approximation of the sign function on the union of two symmetric intervals $[-b,-a]\cup[a,b]$.
\begin{proposition}[{\hspace{1sp}\cite[Theorem~1]{EremenkoYuditskii07}}]
	\label[proposition]{prop:sign-polynomial-approx-bound--hasson}
	Let $0 < a < b$. We have
	\begin{equation*}
		\lim_{m \to \infty} \sqrt{m} \bigg(\frac{b +  a}{b - a}\bigg)^m E_{2m+1}(\sign, [-b, -a] \cup [a, b]) = \frac{b - a}{\sqrt{\pi ab}}.
	\end{equation*}
\end{proposition}
The statement of \cref{prop:sign-polynomial-approx-bound--hasson} implies that there exists $C > 0$ such that
\begin{equation*}
	E_{2m+1}(\sign, [-b, -a] \cup [a, b]) \le \frac{C}{\sqrt{m}} \bigg(\frac{b -  a}{b + a}\bigg)^m,
\end{equation*}
and for $m$ large enough we can take
\begin{equation*}
	C = \frac{b - a}{\sqrt{\pi ab}} + 1.
\end{equation*}
\begin{remark}
	In a situation in which the spectrum of $A$ is not symmetric with respect to the origin, using the bound in \cref{prop:sign-polynomial-approx-bound--hasson} to bound the error of the Lanczos method for the sign function can yield overly pessimistic results, since it would require us to consider intervals that are much wider than~$\spec(A)$.
	In such a case, it would be more convenient to consider the polynomial approximation of the sign function on a nonsymmetric set. 
	For instance, the asymptotics of $E_{2k+1}(\sign, [-a, -1] \cup [1, b])$ for~$a \ne b$ have been studied in \cite{EremenkoYuditskii11}, where the authors generalize \cref{prop:sign-polynomial-approx-bound--hasson} to \cite[Theorem~1.1]{EremenkoYuditskii11}. The result in \cite{EremenkoYuditskii11} could be used to obtained more refined error bounds for the Lanczos approximation error, but the complexity of the statement, which involves Green's functions and integral representations, makes it significantly more difficult to employ in our setting.
\end{remark}

We can easily turn \cref{prop:sign-polynomial-approx-bound--hasson} into a bound for the step function $h_\mu$ for any~$\mu \in \R$. Let us define
\begin{equation*}
	b_1 := \lambda_1 - \mu,  \qquad 	a_1 := \lambda_{\eigcount(\mu)} - \mu, 
	\qquad a_2 := \lambda_{\eigcount(\mu) + 1} - \mu, \qquad b_2 := \lambda_n - \mu,
\end{equation*}
and let
\begin{equation*}
	a := \min\{ \abs{a_1}, \abs{a_2} \}
	\qquad \text{and} \qquad
	b := \max \{ \abs{b_1}, \abs{b_2} \}.
\end{equation*}
Note that $\spec(A) \subset[\lambda_1, \lambda_{\eigcount(\mu)}] \cup [\lambda_{\eigcount(\mu) + 1}, \lambda_n]$ and recall that $h_\mu(x) = -\frac{1}{2} \sign(x - \mu) + \frac{1}{2}$, so we have
\begin{align*}
	E_{2m+1}(h_\mu, [\lambda_1, \lambda_{\eigcount(\mu)}] \cup [\lambda_{\eigcount(\mu) + 1}, \lambda_n]) & = E_{2m+1}(h_0, [b_1, a_1] \cup [a_2, b_2]) \\
	& = \frac{1}{2} E_{2m+1}(\sign, [b_1, a_1] \cup [a_2, b_2]) \\
	& \le \frac{1}{2} E_{2m+1}(\sign, [-b, -a] \cup [a, b])  \\
	& \le \frac{C}{2\sqrt{m}} \bigg(\frac{b - a}{b + a}\bigg)^{m},
\end{align*}
with $C = 1 \,+\, (b - a)/\sqrt{\pi ab}$, where we used \cref{prop:sign-polynomial-approx-bound--hasson} for the last inequality.
Denoting the relative spectral gap associated to $\mu$ by $\theta := a / b$, we obtain
\begin{equation}
	\label{eqn:stepfunction-polynomial-approximation-bound}
	E_{2m+1}(h_\mu, [\lambda_1, \lambda_{\eigcount(\mu)}] \cup [\lambda_{\eigcount(\mu) + 1}, \lambda_n]) \le \frac{C}{2\sqrt{m}} \bigg(\frac{1 - \theta}{1 + \theta}\bigg)^{m}.
\end{equation}
We can combine this result with \cref{eqn:lanczos-matfun-approx--basic-bound} to bound the error in the Lanczos approximation of quadratic forms with $P_\mu = h_\mu(A)$, provided that no eigenvalue of $T_m$ lies within the gap $(\lambda_{\eigcount(\mu)}, \lambda_{\eigcount(\mu) + 1})$.
Unfortunately, \rev{we cannot exclude a priori that an eigenvalue of $T_m$ lies within the gap, and indeed this often happens in practice}; however, since there must be at least one eigenvalue of~$A$ in the open interval between two consecutive Ritz values \cite[Section~4]{Kuijlaars06}, there can be at most one eigenvalue of $T_m$ in the interval $(\lambda_{\eigcount(\mu)}, \lambda_{\eigcount(\mu)+1})$. This fact can be proved using the theory of orthogonal polynomials, see, for instance, \cite[Theorem~1.21]{Gautschi04}, using the fact that the Ritz values are the zeros of an orthogonal polynomial associated with a Gauss quadrature rule; see, e.g., \cite[Theorem~6.2]{GolubMeurant10}.
In the analysis that follows, we assume for simplicity that there are no Ritz values in the gap $(\lambda_{\eigcount(\mu)}, \lambda_{\eigcount(\mu) + 1})$. We discuss how we can modify our approach to deal with this issue in \cref{rem:ritz-value-in-gap}.

\begin{remark}
	\label{rem:krylov-subspace-A2}
	In the literature, several techniques have been used to obtain approximations to the sign matrix function that circumvent the problem of Ritz values getting arbitrarily close to zero \cite{EFLSV02}, for instance by writing $\sign(z) = z \cdot (z^2)^{-1/2}$ and constructing an approximation of $(A^2)^{-1/2}\vec x$ using the Krylov subspace $\kryl_m(A^2, \vec x)$, which ensures that the eigenvalues of the projected matrix are bounded away from $0$, since $A^2$ is positive definite as long as $A$ is nonsingular. Alternatively, one can construct an approximation using harmonic Ritz values \cite{PPV95}, which are always outside of the interval between the largest negative and the smallest positive eigenvalue of $A$, so they avoid the discontinuity of the sign function.
	Since in this work our goal is to approximate $\vec x^T h_\mu(A) \vec x$ for several $\mu$, we opted to use the standard Lanczos approximation~\cref{eqn:quadform-lanczos-approximation}, for which it is very cheap and straightforward to compute approximations for many different functions.
\end{remark}

Combining \cref{eqn:lanczos-matfun-approx--basic-bound} with \cref{eqn:stepfunction-polynomial-approximation-bound}, we obtain the following bound on the Lanczos approximation error.
\begin{proposition}
	\label[proposition]{prop:lanczos-matfun-approx--apriori-bound}
	Let $\quadform_m$ be the Lanczos approximation \cref{eqn:quadform-lanczos-approximation} to $\vec x^T h_\mu(A) \vec x$ from the Krylov subspace $\kryl_m(A, \vec x)$, and assume that there is no eigenvalue of $T_m$ in the interval $(\lambda_{\eigcount(\mu)}, \lambda_{\eigcount(\mu)+1})$. We have
	\begin{equation}
		\label{eqn:lanczos-matfun-approx--gap-bound-noritzval}
		\abs{\vec x^T h_\mu(A) \vec x - \quadform_m} \le \frac{C_\rev{\theta}\norm{\vec x}_2^2}{\sqrt{m-1}} \bigg( \frac{1 - \theta}{1 + \theta} \bigg)^{m-1},
	\end{equation}
	where $\theta$ is defined above \rev{and $C_\theta = 1 + (1-\theta)/\sqrt{\pi \theta}$}.
\end{proposition}

Note that if we take a different value of $\mu$ in the interval $(\lambda_{\eigcount(\mu)}, \lambda_{\eigcount(\mu) + 1})$, the error $\abs{\vec x^T h_\mu(A) \vec x - \quadform_m}$ remains the same, because the function $h_\mu$ takes the same values on the eigenvalues of $A$ and~$T_m$ and hence the matrix functions $h_\mu(A)$ and $h_\mu(T_m)$ do not change. On the other hand, the bound in \cref{prop:lanczos-matfun-approx--apriori-bound} depends on $\theta$, which in turn depends on $\mu$. It is quite easy to see that \rev{the sharpest bound in \cref{eqn:lanczos-matfun-approx--gap-bound-noritzval} is obtained when $\theta$ is as large as possible, which occurs} when $\mu$ is precisely in the middle of the gap, i.e. $\mu = \frac{1}{2}(\lambda_{\eigcount(\mu)} + \lambda_{\eigcount(\mu)+1})$. 

The convergence predicted by \cref{prop:lanczos-matfun-approx--apriori-bound} is at least exponential with rate~$(1-\theta)/(1+\theta)$, with in addition an $O(1/\sqrt{m})$ factor which is independent of the gap width. Note that $C_\rev{\theta}$ grows as the gap becomes smaller; see \cref{prop:sign-polynomial-approx-bound--hasson}. 
In order to have an error $\abs{\vec x^T h_\mu(A) \vec x - \quadform_m} \le \varepsilon$, the number of Lanczos iterations $m$ should satisfy
\begin{equation}
	\label{eqn:lanczos-accuracy--cost-requirement}
	m + \frac{\log(m-1)}{2\log\big((1+\theta)/(1-\theta)\big)}  \ge 1 + \frac{\log\big(C_\rev{\theta} \norm{\vec x}_2^2 / \varepsilon\big)}{\log\big((1+\theta)/(1-\theta)\big)}.
\end{equation}
Assuming that the relative gap $\theta$ is large enough, we do not lose much by discarding the $\log(m-1)$ term, and we can take
\begin{equation}
	\label{eqn:lanczos-accuracy--largegap-cost-requirement}
	m \ge 1 + \frac{\log\big(C_\rev{\theta} \norm{\vec x}_2^2 / \varepsilon\big)}{\log\big((1+\theta)/(1-\theta)\big)}.
\end{equation}
On the other hand, in a situation where the spectral gap is very small, the Lanczos method effectively converges like $O(1/\sqrt{m})$. Note that if $\vec x$ is a Gaussian random vector we have $\expec[\norm{\vec x}_2^2] = n$, so with a very small gap for which $\theta \approx 0$ we would need to take
\begin{equation*}
	m \gtrsim \frac{C_\rev{\theta}^2 n^2}{\varepsilon^2},
\end{equation*}
which is definitely unfeasible. This means that when using the Lanczos method we can only hope to approximate $\vec x^T h_\mu(A) \vec x$ accurately when $\mu$ is inside a relatively large gap in the spectrum of $A$, and we should instead expect to have quite a large error whenever $\mu$ is close to an eigenvalue of $A$. 

A crucial requirement for using the bound in \cref{prop:lanczos-matfun-approx--apriori-bound} is the knowledge of the relative gap width $\theta$. Since the gaps in the spectrum are not known to us in advance, we cannot use this result to check if the error of the Lanczos method is below a certain tolerance $\varepsilon$.
On the other hand, \cref{prop:lanczos-matfun-approx--apriori-bound} can be useful to choose the number $m$ of Lanczos iterations at the start of the algorithm. 
Indeed, if our goal is to find all the gaps in the spectrum of~$A$ with relative width at least~$\theta$, for any given $\varepsilon$ we can select $m$ using \cref{eqn:lanczos-accuracy--largegap-cost-requirement} to ensure that the error $\abs{\vec x^T h_\mu(A) \vec x - \quadform_m} \le \varepsilon$ whenever~$\mu$ is inside a gap with width at least $\theta$. However, we are still unable to determine for which values of~$\mu$ this condition is satisfied.
The final component that we need to make the algorithm work is an a posteriori bound on the error, which will allow us to detect when $\abs{\vec x^T h_\mu(A) \vec x - \quadform_m} \le \varepsilon$, without knowing the gaps in the spectrum in advance. For this purpose we can use an a posteriori error bound from \cite{BRS23}, which we adapt to the step function $h_\mu$ in \cref{subsec:lanczos-a-posteriori-bound}.

\begin{remark}
	\label[remark]{rem:ritz-value-in-gap}
	If there is an eigenvalue of $T_m$ inside the gap $(\lambda_{\eigcount(\mu)}, \lambda_{\eigcount(\mu) + 1})$, the results above do not hold. \rev{However, it is still possible to detect that there is a gap in the spectrum of $A$.} \rev{Indeed, let us assume that the gap $(\lambda_{\eigcount(\mu)}, \lambda_{\eigcount(\mu) + 1})$ has relative width $\theta$ and that there is exactly one eigenvalue $\tau$ of $T_m$ in the gap. Then, the gap is split into the two smaller gaps $(\lambda_{\eigcount(\mu)}, \tau)$ and $(\tau, \lambda_{\eigcount(\mu)})$, and one of then has a relative width at least $\theta/2$.} In other words, there is an interval of relative width~$\theta/2$ where there are no eigenvalues of $A$ or $T_m$, and hence when $\mu$ is inside that interval the Lanczos approximation error satisfies the bound \cref{eqn:lanczos-matfun-approx--gap-bound-noritzval} with $\theta$ replaced by~$\theta/2$. \rev{In order to ensure that this gap is detected, it is then enough to use $\theta/2$ instead of $\theta$ as a target gap width for our algorithm. We remark that with this approach we are only able to detect the larger portion of the gap, and not the whole gap in $\spec(A)$; however, the accuracy in the approximation of the gap can be refined for a small computational cost by simply running a few extra Lanczos iterations, which would cause the eigenvalue of $T_m$ to move around the gap or, even better, to disappear from the gap.}
\end{remark}

\subsection{An a posteriori error bound for the Lanczos algorithm}
\label{subsec:lanczos-a-posteriori-bound}

In this section we adapt an a posteriori error bound given in~\cite{BRS23} for the approximation of~$\vec x^T f(A) \vec x$ with a rational Krylov method to the case of the Lanczos approximation of $f = h_\mu$. 
The result given in \cite[Theorem~4.8]{BRS23} employs the Cauchy integral formula and the residue theorem to obtain upper and lower bounds on the error in the approximation of $\vec x^T f(A) \vec x$ with a rational Krylov method. This result can be easily used to bound the error in the Lanczos approximation for $f = h_\mu$, but it requires some caution since the function~$h_\mu$ is discontinuous at $\mu$. In the following, we briefly retrace the proof of \cite[Theorem~4.8]{BRS23} and adapt it to the current setting. 

If $\Gamma \subset \C$ is a simple closed curve that surrounds the eigenvalues of $A$ below $\mu$, we have
\begin{equation*}
	P_\mu = h_\mu(A) = \frac{1}{2 \pi i} \int_\Gamma (z I - A)^{-1} \de z.
\end{equation*}
If the eigenvalues of $T_m$ that are smaller than $\mu$ also lie in the interior of $\Gamma$, we have
\begin{equation*}
	h_\mu(A) \vec x - \norm{\vec x}_2 V_m h_\mu(T_m) \vec e_1 = \frac{1}{2 \pi i} \int_\Gamma (z I - A)^{-1} \vec r_m(z) \de z,
\end{equation*}
where
\begin{equation*}
	\vec r_m(z) = \vec x - (z I - A) \vec x_m(z), \qquad \text{with} \qquad \vec x_m(z) = \norm{\vec x}_2 V_m (z I - T_m)^{-1} \vec e_1,
\end{equation*}
that is, $\vec r_m(z)$ is the residual of a shifted linear system after $m$ iterations of the Lanczos method. With the same algebraic manipulations used in \cite[Section~4.3]{BRS23}, we obtain
\begin{equation*}
	\vec x^T h_\mu(A) \vec x - \norm{\vec x}_2^2 \vec e_1^T h_\mu(T_m) \vec e_1 = \frac{\norm{\vec x}_2^2}{2 \pi i} \int_\Gamma \varphi_m(z)^2 \vec v_{m+1}^T (z I -A)^{-1} \vec v_{m+1} \de z,
\end{equation*}
with
\begin{equation*}
	\varphi_m(z) = t_{m+1, m} \vec e_m^T (z I - T_m)^{-1} \vec e_1,
\end{equation*}
where $t_{m+1, m}$ denotes the element in position $(m+1, m)$ of the matrix $\underline{T}_m$.
Using an eigenvalue decomposition $T_m = U_m D_m U_m^T$, with $U_m$ orthogonal and $D_m = \diag(\theta_1, \dots, \theta_m)$, we can define
\begin{equation*}
	\begin{bmatrix}
		\alpha_1, \dots, \alpha_m
	\end{bmatrix} := t_{m+1,m} \vec e_m^T U_m, \qquad \begin{bmatrix}
		\beta_1, \dots, \beta_m
	\end{bmatrix}^T := U_m^T \vec e_1
\end{equation*}
and 
\begin{equation*}
	\gamma_j := \sum_{k \,:\, k \ne j} \alpha_k \beta_k \frac{1}{\theta_j - \theta_k}, \qquad j = 1, \dots, m.
\end{equation*}
With these definitions, the function $\varphi_m(z)^2$ can be rewritten in the form
\begin{equation*}
	\varphi_m(z)^2 = \sum_{j = 1}^m \alpha_j^2 \beta_j^2 \frac{1}{(z - \theta_j)^2} + 2 \sum_{j = 1}^m \alpha_j \beta_j \gamma_j \frac{1}{z - \theta_j}.
\end{equation*}
Using the residue theorem, with the same procedure used in \cite[Section~4.3]{BRS23} we obtain
\begin{equation}
	\label{eqn:lanczos-error--residue-formulation}
	\vec x^T h_\mu(A) \vec x - \norm{\vec x}_2^2 \vec e_1^T h_\mu(T_m) \vec e_1 = \norm{\vec x}_2^2 \vec v_{m+1}^T g_m(A) \vec v_{m+1},
\end{equation}
with
\begin{equation}
	\label{eqn:lanczos-post-bound--gm-definition}
	g_m(z) := \begin{cases}
		\dsum_{j \,:\, \theta_j > \mu} \alpha_j^2 \beta_j^2 \dfrac{1}{(z - \theta_j)^2} + 2 \alpha_j \beta_j \gamma_j \dfrac{1}{z - \theta_j}   & \text{if $z < \mu$,} \\
		- \dsum_{j \,:\, \theta_j < \mu} \alpha_j^2 \beta_j^2 \dfrac{1}{(z - \theta_j)^2} - 2 \alpha_j \beta_j \gamma_j \dfrac{1}{z - \theta_j} & \text{if $z > \mu$.}
	\end{cases}
\end{equation}
We can use \cref{eqn:lanczos-error--residue-formulation} to obtain the following upper bound for the Lanczos approximation error, which corresponds to the upper bound in \cite[Theorem~4.8]{BRS23}.

\begin{proposition}
	\label[proposition]{prop:lanczos-a-posteriori-error-bound}
	Let $\quadform_m$ be the Lanczos approximation \cref{eqn:quadform-lanczos-approximation} to $\vec x^T h_\mu(A) \vec x$ from the Krylov subspace $\kryl_m(A, \vec x)$. We have
	\begin{equation*}
		\abs{\vec x^T h_\mu(A) \vec x - \quadform_m} \le \norm{\vec x}_2^2 \max_{z \in \spec(A)} {\abs{g_m(z)}} \le \norm{\vec x}_2^2 \max_{z \in [\lambda_{\min}, \lambda_{\max}]} {\abs{g_m(z)}},
	\end{equation*}
	where $g_m$ is defined in \cref{eqn:lanczos-post-bound--gm-definition}.
\end{proposition}

In contrast with what happens when the function $f$ is continuous over the whole spectral interval of $A$, in this case the function $g_m(z)$ is discontinuous at $\mu$ and the lower bound from \cite[Theorem~4.8]{BRS23} no longer holds. Indeed, if we denote by $(\lambda_k, \vec u_k)$ the normalized eigenpairs of the matrix $A$, we can write
\begin{equation*}
	\vec v_{m+1}^T g_m(A) \vec v_{m+1} = \sum_{k = 1}^m g_m(\lambda_k) (\vec u_k^T \vec v_{m+1})^2, \qquad \text{with} \qquad \sum_{k = 1}^m (\vec u_k^T \vec v_{m+1})^2 = 1.
\end{equation*}
This expression shows that the quadratic form $\vec v_{m+1}^T g_m(A) \vec v_{m+1}$ can potentially take any value in the convex hull of the set $\{ g_m(\lambda_k) \}_{k = 1}^m$, depending on the coefficients $(\vec u_k^T \vec v_{m+1})^2$. In the case that $g_m(\lambda_k)$ takes both positive and negative values, it is possible to have $\vec v_{m+1}^T g_m(A) \vec v_{m+1} = 0$, so the lower bound
\begin{equation*}
	\abs{\vec x^T h_\mu(A) \vec x - \norm{\vec x}_2^2 \vec e_1^T h_\mu(T_m) \vec e_1} \ge \norm{\vec x}_2^2 \min_{z \in [\lambda_{\min}, \lambda_{\max}]} {\abs{g_m(z)}}
\end{equation*}
may not hold. Note that this argument is not in contradiction with the statement of \cite[Theorem~4.8]{BRS23} when the function $g_m$ is continuous over the whole spectral interval of $A$, because if $g_m(\lambda_k)$ takes both negative and positive values there must be a point $z \in [\lambda_{\min}, \lambda_{\max}]$ such that $g_m(z) = 0$ by continuity of $g_m$.

\begin{example}
	\label{example:lanczos-bound-comparison}
	In this example we compare the error of the Lanczos approximation of $\vec x^T h_\mu(A) \vec x$ with the a priori error bound from \cref{prop:lanczos-matfun-approx--apriori-bound}, the a posteriori error bound form \cref{prop:lanczos-a-posteriori-error-bound}, and the simple error estimate $\abs{\vec x^T h_\mu(A) \vec x - \quadform_m} \approx \abs{\quadform_m - \quadform_{m+1}}$. 
	We consider a symmetric matrix~$A$ of size $n = 1000$ with $\spec(A) \subset [0, 40]$, and gaps in the intervals $[10, 11]$ and $[20, 25]$.  
	We plot the convergence of the Lanczos method in \cref{fig:example-lanczos-bound}, where on the left plot $\mu = 22$ is contained in the larger gap, and on the right plot $\mu = 10.6$ is inside the smaller gap. In the legend, the ``central $\mu$'' a priori bound refers to the bound obtained by taking $\mu$ in the middle of the gap, as mentioned below \cref{prop:lanczos-matfun-approx--apriori-bound}.

	Note that the convergence of the Lanczos method is much slower when $\mu$ is inside the smaller gap, and in both cases the convergence behavior is irregular and oscillating. 
	The ``central $\mu$'' is the most accurate between the two a priori bounds, but it does not capture the correct convergence rate for $\mu = 10.6$, which is further from the middle of the spectrum; in this case, a better a priori bound would be obtained by using a convergence result with nonsymmetric intervals, such as \cite[Theorem~1.1]{EremenkoYuditskii11}. 
	The a~posteriori bound is also quite irregular, but it is able to capture the convergence rate of the Lanczos approximation correctly in both cases, although it is off by a constant factor. 
	Surprisingly, the error estimate $\abs{\vec x^T h_\mu(A) \vec x - \quadform_m} \approx \abs{\quadform_m - \quadform_{m+1}}$ is very accurate, and it seems to almost never underestimate the error. This behavior can be explained in the following way: the error estimate based on consecutive differences satisfies the triangle inequality
	\begin{equation*}
		\abs{\quadform_m - \quadform_{m+1}} \ge \abs*{\, \abs{\vec x^T h_\mu(A) \vec x - \quadform_m} - \abs{\vec x^T h_\mu(A) \vec x - \quadform_{m+1}} \, },
	\end{equation*}
	and due to the oscillations in the convergence of the Lanczos approximation, the two right-hand side terms $\abs{\vec x^T h_\mu(A) \vec x - \quadform_{m}}$ and $\abs{\vec x^T h_\mu(A) \vec x - \quadform_{m+1}}$ often have significantly different moduli, so we have
	\begin{equation}
		\label{eqn:consec-diff-est-accuracy--oscillations}
		\abs*{\, \abs{\quadform - \quadform_m} - \abs{\quadform - \quadform_{m+1}} \, } \ge c \max\left\{ \abs{\quadform - \quadform_m}, \abs{\quadform - \quadform_{m+1}} \right\}, \qquad \quadform := \vec x^T h_\mu(A) \vec x,
	\end{equation}
	for a constant $c$ that is close to $1$. For instance, if $\abs{\quadform - \quadform_{m}} > \abs{\quadform - \quadform_{m+1}}$, we can take  
	\begin{equation*}
		c = 1 - \frac{\abs{\quadform - \quadform_{m+1}}}{\abs{\quadform - \quadform_{m}}}.
	\end{equation*}
	The inequality \cref{eqn:consec-diff-est-accuracy--oscillations} explains why the consecutive difference estimate is so often an upper bound when the Lanczos method has oscillating convergence, as in the current setting. 

	\begin{figure}[t]
		\centering
		\makebox[\linewidth][c]{
			\begin{subfigure}[t]{.5\textwidth}
				\includegraphics[width=\textwidth]{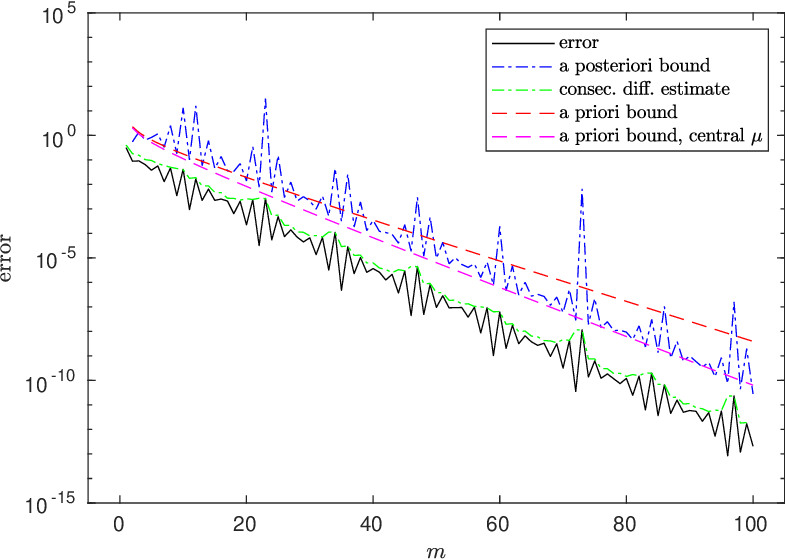}
			\end{subfigure}
			\hspace{0.5cm}
			\begin{subfigure}[t]{.5\textwidth}
				\includegraphics[width=\textwidth]{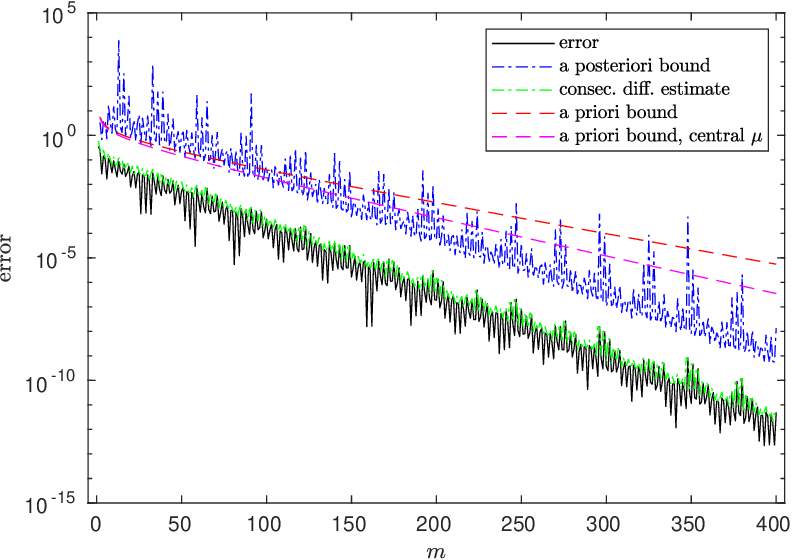}
			\end{subfigure}
			}
		\caption[Comparison of a priori and a posteriori bounds for Lanczos approximation of $\vec x^T P_\mu \vec x$]{Approximation of $\vec x^T h_\mu(A) \vec x$ with the Lanczos algorithm, and comparison with the a priori error bound from \cref{prop:lanczos-matfun-approx--apriori-bound}, the a posteriori error bound from \cref{prop:lanczos-a-posteriori-error-bound} and the consecutive difference estimate. The spectrum of $A$ has gaps in the intervals $[10, 11]$ and $[20, 25]$.  Left: $\mu = 22$. Right: $\mu = 10.6$.  
		\label{fig:example-lanczos-bound}}
	\end{figure}
\end{example}

\section{Final algorithm}
\label{sec:final-algorithm}

In this section we put together all the theoretical results obtained in \cref{sec:algorithm-analysis} to improve \cref{algorithm:eigenvalue-gap-finder--basic} and obtain an efficient algorithm that also provides some guarantees on its output.

\subsection{Choosing the input parameters}
\label{subsec:choosing-input-parameters}

First of all, we need a criterion to decide the number $s$ of Hutchinson sample vectors and the number~$m$ of Lanczos iterations. For this purpose, given a failure probability $\delta$ and a target relative gap width~$\theta$, we set the goal to find all gaps in the spectrum of $A$ with relative width greater than or equal to $\theta$, with probability~$1-\delta$.

How should we choose $s$ and $m$ so that all these gaps are found?
Let us consider a candidate gap $[\mu_1, \mu_2] \subset \R$ and fix a parameter $\varepsilon > 0$, and say that we take $m$ such that the Lanczos approximation error for each quadratic form $\vec x_i^T P_\mu \vec x_i$ is smaller than $\varepsilon/2$ for all~$\mu$ inside the gap, so that $\traceH{s}(P_\mu)$ is contained in an interval of width smaller than $\varepsilon$ for all $\mu \in [\mu_1, \mu_2]$. Then there are two possible scenarios: either there are no eigenvalues of $A$ in the interval $[\mu_1, \mu_2]$, or all the jumps in $\traceH{s}(P_\mu)$ associated to eigenvalues inside that interval are $\varepsilon$-small (more precisely, the sum of the heights of these jumps is smaller than $\varepsilon$). This second event has a small probability of occurring, which can be bounded using \cref{lemma:small-jump-probability}, so we can conclude that $[\mu_1, \mu_2]$ is a gap in the spectrum of $A$ with high probability.

The probability that all jumps in $\traceH{s}(P_\mu)$ associated to eigenvalues in $[\mu_1, \mu_2]$ are $\varepsilon$-small is smaller than the probability of a single jump being $\varepsilon$-small, so we can impose
\begin{equation*}
	\prob(y_s \le s \varepsilon) \le \delta,
\end{equation*}
where $y_s$ is a $\chi^2$ distribution with $s$ degrees of freedom, see \cref{subsec:algorithm-analysis--hutchinson}.
By \cref{lemma:small-jump-probability}, this condition is satisfied provided that
\begin{equation*}
	s \ge \frac{2 \log(1/\delta)}{\log(1/\varepsilon) + \varepsilon - 1}.
\end{equation*}
If the relative width of $[\mu_1, \mu_2]$ is at least $\theta$, by \cref{prop:lanczos-matfun-approx--apriori-bound} we have
\begin{equation*}
	\Big\lvert \traceH{s}(P_\mu) - \frac{1}{s}\sum_{i = 1}^s\quadform_m^{(i)}\Big\rvert \le \frac{C_\rev{\theta} \sum_{i = 1}^s \norm{\vec x_i}_2^2}{s\sqrt{m-1}} \bigg( \frac{1 - \theta}{1 + \theta} \bigg)^{m-1},
\end{equation*}
where we denote by $\quadform_m^{(i)}$ the approximation of $\vec x_i^T P_\mu \vec x_i$ after $m$ Lanczos iterations. By \cref{eqn:lanczos-accuracy--largegap-cost-requirement}, the above bound is smaller than $\varepsilon/2$ if we take
\begin{equation}
	\label{eqn:lanczos-iterations-m-best-choice}
	m \ge 1 + \frac{\log\big(2 C_\rev{\theta} \sum_{i = 1}^s\norm{\vec x_i}_2^2 / s \varepsilon\big)}{\log\big((1+\theta)/(1-\theta)\big)}, \qquad C_\rev{\theta} = \frac{1 - \theta}{\sqrt{\pi \theta}} + 1.
\end{equation}
Note that although the a priori bound in \cref{prop:lanczos-matfun-approx--apriori-bound} allows us to say that the Lanczos error is smaller than $\varepsilon/2$ when $\mu$ is inside a gap wider than $\theta$, we cannot use it to determine for which values of $\mu$ this actually happens. To detect the gaps in practice, we are going to use either the a posteriori bound from \cref{prop:lanczos-a-posteriori-error-bound} or the error estimate based on consecutive differences, and assume that it is at least as accurate as the a priori bound, i.e., that it is also smaller than $\varepsilon/2$ inside a gap with relative width $\theta$. This assumption is usually satisfied, as shown in \cref{example:lanczos-bound-comparison}, at least in the central portion of each gap. \rev{Recalling the discussion in \cref{rem:ritz-value-in-gap}, we can replace $\theta$ with $\theta/2$ to ensure that all gaps of width $\theta$ are detected even when there is a Ritz value within a gap.}

We have obtained bounds on $s$ and $m$ that depend on $\delta$, $\theta$ and $\varepsilon$, where the latter is a parameter that can be freely chosen. In order to obtain an algorithm that is as efficient as possible, we should aim to minimize the term that dominates the computational cost of the algorithm, which is the $O(sm \, \nnz(A))$ term corresponding to the $sm$ matrix-vector products with $A$. We have
\begin{equation*}
	sm \ge \frac{2 \log(2/\delta)}{\log(1/\varepsilon) + \varepsilon - 1} \cdot \bigg( 1 + \frac{\log\big(2 C_\rev{\theta} \sum_{i = 1}^s \norm{\vec x_i}_2^2 / s \varepsilon\big)}{\log\big((1+\theta)/(1-\theta)\big)} \bigg),
\end{equation*}
and $\varepsilon$ should be chosen in order to minimize the right hand side.
With some simple computations one can show that the function
\begin{equation*}
	f(\varepsilon) = \frac{1}{\log(1/\varepsilon) + \varepsilon - 1} \cdot \bigg( 1 + \frac{\log\big(2 C_\rev{\theta} n / \varepsilon\big)}{\log\big((1+\theta)/(1-\theta)\big)} \bigg)
\end{equation*}
is increasing for $\varepsilon \in (0, 1)$ since $2 C_\rev{\theta} n \ge 1$. Recalling that $\expec[\norm{\vec x_i}^2] = n$ for all $i$, this implies that we should take $\varepsilon$ as small as possible, and since $s$ must be an integer, in practice the most efficient choice is to take $\varepsilon$ so that the probability of a jump being $\varepsilon$-small is lower than~$\delta$ already when $s = 1$, i.e.,
\begin{equation}
	\label{eqn:epsilon-best-choice}
	\varepsilon \le e^{-1} \delta^2.
\end{equation}
In other words, we should simply use a single quadratic form $\vec x^T P_\mu \vec x$ to estimate $\trace(P_\mu)$, instead of using the more general stochastic trace estimator with $s$ different random vectors.   

\rev{
		We mention that bounds on the number of Hutchinson sample vectors required to achieve a certain accuracy have been obtained in \cite{CTU21,BKM22} in the context of spectral density estimation. For instance, it is shown in \cite[Theorem~1]{CTU21} that, in order to obtain an accuracy of order $t$ on the normalized spectral density in the Wasserstein distance, it is enough to take roughly $\ceil{1/nt^2}$ sample vectors. When $t$ is not too small, this implies that it is sufficient to use a single vector in Hutchinson's estimator, a result that matches the conclusion reached in this section. However, note that the result from \cite{CTU21} is not applicable in our setting: indeed, an absolute error $\varepsilon$ on $\eigcount(\mu)$ corresponds to an error $t = \varepsilon/n$ on the normalized spectral density, so that the number of samples predicted by \cite[Theorem~1]{CTU21} would be roughly $\ceil{n/\varepsilon^2}$, which grows linearly with~$n$, matching our initial observations in \cref{subsec:algorithm-analysis--hutchinson}. Another key difference is that the theoretical results for spectral density estimation are usually given in terms of the Wasserstein distance or some other global measure for the quality of approximation. On the other hand, we are only interested in approximating the spectral density close to large gaps, so it is reasonable to expect that better results can be obtained with a different kind of analysis. We refer to \cref{subsec:literature-comparison} for further discussion and comparison of our analysis with the existing literature.
}

\subsection{Detecting gaps in the spectrum}
\label{subsec:detecting-gaps}

To detect the gaps in the spectrum, we use an a posteriori error bound to determine when the error of the Lanczos approximation is smaller than~$\varepsilon/2$. We focus on the case $s = 1$, with a single random Gaussian vector $\vec x \in \R^n$, since we concluded that it is the most efficient choice in \cref{subsec:choosing-input-parameters}. More precisely, for each~$\mu$ we perform $m$ iterations of the Lanczos method to obtain $\quadform_m(\mu) \approx \vec x^T P_\mu \vec x$ with the approximation~\cref{eqn:quadform-lanczos-approximation}, and we use either \cref{prop:lanczos-a-posteriori-error-bound} or the estimate based on consecutive differences to obtain
\begin{equation*}
	\abs{\vec x^T P_\mu \vec x - \quadform_m(\mu)} \le \postbound_m(\mu),
\end{equation*}
where either 
\begin{equation*}
	\postbound_m(\mu) = \norm{\vec x}_2^2 \max_{z \in [\lambda_{\min}, \lambda_{\max}]} \abs{g_m(z)} \qquad \text{with $g_m(z)$ as defined in \cref{eqn:lanczos-post-bound--gm-definition}},
\end{equation*}
or $\postbound_m(\mu) = c \abs{\quadform_m(\mu) - \quadform_{m+1}(\mu)}$, with a safety factor $c \ge 1$. As shown in \cref{example:lanczos-bound-comparison}, the consecutive difference error estimate is usually more accurate than the a posteriori error bound from \cref{prop:lanczos-a-posteriori-error-bound}, and very often we can also expect it to be an actual upper bound due to the oscillation in the convergence of the Lanczos algorithm. In the discussion that follows, we are going to assume that $\postbound_m(\mu)$ is always an upper bound for the error $\abs{\vec x^T P_\mu \vec x - \quadform_m(\mu)}$; since there is no guarantee that this is true for the consecutive difference estimate, we will discuss some ways to make it more robust in \cref{example:trace-bound-final}.
For all $\mu$, letting
\begin{align*}
	\qfupper_m(\mu) & := \quadform_m(\mu) + \postbound_m(\mu), \\
	\qflower_m(\mu) & := \quadform_m(\mu) - \postbound_m(\mu),
\end{align*}
we have $\vec x^T P_\mu \vec x \in [\qflower_m(\mu), \qfupper_m(\mu)]$.

Recalling that the quadratic form $\vec x^T P_\mu \vec x$ is an increasing function of $\mu$, we can refine the upper and lower bounds $\qfupper_m(\mu)$ and $\qflower_m(\mu)$. Indeed, we can also take the upper and lower bounds to be increasing in $\mu$: specifically, for any $\mu < \mu'$ we must have
\begin{equation*}
	\vec x^T P_\mu \vec x \le \vec x^T P_{\mu'} \vec x \le \qfupper_m(\mu') \qquad \text{and} \qquad \vec x^T P_{\mu'} \vec x \ge \vec x^T P_\mu \vec x \ge \qflower_m(\mu),
\end{equation*}
so we can replace the bounds $\qfupper_m(\mu)$ and $\qflower_m(\mu)$ with, respectively,
\begin{equation}
	\label{eqn:lanczos-post-bound-monotone}
	\qfupperhat_m(\mu) := \min_{\mu' \ge \mu} \qfupper_m(\mu') \qquad \text{and} \qquad \qflowerhat_m(\mu) := \max_{\mu' \le \mu} \qflower_m(\mu').
\end{equation}
Note that the bounds $\qfupperhat_m(\mu)$ and $\qflowerhat_m(\mu)$ are by construction increasing functions of $\mu$. We can further improve the bounds by taking the best ones over several values of $m$. Indeed, given a set of positive integers $\mathcal{M} \subset \N$, we can define
\begin{equation}
	\label{eqn:lanczos-upper-lower-a-posteriori-bounds}
	\qfupperhat_\mathcal{M}(\mu) := \min_{m \in \mathcal{M}}\qfupperhat_m(\mu) \qquad \text{and} \qquad \qflowerhat_\mathcal{M}(\mu) := \max_{m \in \mathcal{M}}\qflowerhat_m(\mu),
\end{equation}
and we still have $\vec x^T P_\mu \vec x \in [\qflowerhat_\mathcal{M}(\mu), \qfupperhat_\mathcal{M}(\mu)]$. Using a set $\mathcal{M}$ with more than a single value of $m$ can be useful to improve the bounds in the case when there is an eigenvalue of $T_m$ inside one of the gaps. In such a situation, by considering bounds for $\mathcal{M} = \{m-d, \dots, m-1, m\}$ there is a higher chance that one of the matrices $T_{m-j}$, $j = 0, \dots, d$ has no eigenvalues in the gap, so the bounds $\qfupperhat_\mathcal{M}(\mu)$ and $\qflowerhat_\mathcal{M}(\mu)$ are likely going to be better than $\qfupperhat_m(\mu)$ and $\qflowerhat_m(\mu)$.

After incorporating these improvements, we can conclude that if $\varepsilon$ is chosen according to~\cref{eqn:epsilon-best-choice} and for some $\mu_1 < \mu_2$ we have $\abs{\qfupperhat_\mathcal{M}(\mu_2) - \qflowerhat_\mathcal{M}(\mu_1)} \le \varepsilon$, then there are no eigenvalues of $A$ in the interval $[\mu_1, \mu_2]$ with probability at least $1 - \delta$.

\subsection{Algorithm summary}
\label{subsec:algorithm-summary}

The analysis in the previous sections leads to the following algorithm, whose pseudocode is given in \cref{algorithm:eigenvalue-gap-finder--final}.
Given an input failure probability $\delta$ and target relative gap width $\theta$, we first compute~$\varepsilon$ according to \cref{eqn:epsilon-best-choice} such that the probability of having an $\varepsilon$-small jump at an eigenvalue in Hutchinson's trace estimator is less than~$\delta$ when using a single random vector (i.e., with $s = 1$).
We then use \cref{prop:lanczos-matfun-approx--apriori-bound} to select the number $m$ of Lanczos iterations that ensures that the error of the Lanczos approximation to the quadratic form~$\vec x^T h_\mu(A) \vec x$ is smaller than $\varepsilon/2$ whenever $\mu$ is inside a gap with relative width at least $\theta$. As an alternative, we could fix the number of Lanczos iterations $m$ a priori, but by doing so we would lose the guarantee that all the gaps with width larger than $\theta$ are located. Then, we run $m$ iterations of the Lanczos method to compute the approximation \cref{eqn:quadform-lanczos-approximation} and the upper and lower a posteriori bounds $\qfupperhat_\mathcal{M}(\mu)$ and~$\qflowerhat_\mathcal{M}(\mu)$ defined in \cref{eqn:lanczos-upper-lower-a-posteriori-bounds}, for all input values of $\mu$ and a set $\mathcal{M}$ of integers smaller than or equal to~$m$, such as $\{m-d, \dots, m-1, m\}$, with, e.g., $d = 2$.
All the intervals $[\mu_1, \mu_2]$ such that $\abs{\qfupperhat_\mathcal{M}(\mu_2) - \qflowerhat_\mathcal{M}(\mu_1)} \le \varepsilon$ are hence gaps in the spectrum of $A$ with probability at least $1- \delta$.

\begin{algorithm}[t]
	\caption{Eigenvalue gap finder}
	\label{algorithm:eigenvalue-gap-finder--final}
	\begin{algorithmic}[1]
		\Require Symmetric matrix $A \in \R^{n \times n}$, failure probability $\delta$, target relative gap size $\theta$, vector of test parameters $\vec \mu = [\mu_1, \dots, \mu_{N_f}]$
		\Ensure $\texttt{tr}_j \approx \trace(P_{\mu_j})$ for $j = 1, \dots, N_f$, intervals $[a_i, b_i]$ that contain no eigenvalues of~$A$ with probability $1 - \delta$ each

		\State Compute $\varepsilon$ according to \cref{eqn:epsilon-best-choice}, ensuring that the probability of a jump being smaller than $\varepsilon$ with $s = 1$ Hutchinson vectors is smaller than $\delta$.
		\State Draw a random Gaussian vector $\vec x \in \R^n$.
		\State Compute $m$ according to \cref{eqn:lanczos-iterations-m-best-choice}, ensuring that the Lanczos approximation \cref{eqn:quadform-lanczos-approximation} to $\vec x^T P_\mu \vec x$ has an error smaller than $\varepsilon/2$ if $\mu$ is within a gap of relative width $\theta$ or larger.
		\State Select a set $\mathcal{M}$ of integers smaller than or equal to $m$, such as $\mathcal{M} = \{ m-4, \dots, m-1, m \}$.
		\State Compute $V_{m}$ and $\underline{T}_m$ by running $m$ iterations of the Lanczos algorithm with the matrix $A$ and the vector $\vec x$.
		\State Compute the eigendecomposition $T_m = U_m D_m U_m^T$ and $\vec w_m = \norm{\vec x}_2 U_m^T \vec e_1$.
		\For {$j = 1, \dots, N_f$}
		\State Compute $\quadform_{m,j} = \vec w_m^T h_{\mu_j}(D_m)\vec w_m$.
		\For {$m' \in \mathcal{M}$}
		\State 
		\parbox[t]{\dimexpr\linewidth-3em}%
		{
			Compute the monotone upper and lower bounds $\qfupperhat_{m'}(\mu_j)$ and $\qflowerhat_{m'}(\mu_j)$ using \cref{eqn:lanczos-post-bound-monotone}, where $\postbound_{m'}(\mu_j)$ is obtained either with \cref{prop:lanczos-a-posteriori-error-bound} or by running Lanczos for one additional iteration to compute $\abs{\quadform_{m', j} - \quadform_{m'+1,j}}$.
		}
		\EndFor
		\State Compute the bounds $\qfupperhat_\mathcal{M}(\mu_j) = \displaystyle\min_{m' \in \mathcal{M}}\qfupperhat_{m'}(\mu_j)$ and $\qflowerhat_\mathcal{M}(\mu_j) = \displaystyle\max_{m' \in \mathcal{M}}\qflowerhat_{m'}(\mu_j)$.
		\EndFor
		\State Compute $\qfupperhat_{\mathcal{M}}(\vec \mu) - \qflowerhat_{\mathcal{M}}(\vec \mu)$ and determine the largest intervals $[a_i, b_i]$ such that $\abs{\qfupperhat_{\mathcal{M}}(b_i) - \qflowerhat_{\mathcal{M}}(a_i)} \le \varepsilon$.
	\end{algorithmic}
\end{algorithm}

Note that $m$ is selected so that the a priori error bound on the approximation of $\vec x^T h_\mu(A) \vec x$ from \cref{prop:lanczos-matfun-approx--apriori-bound} is smaller than $\varepsilon/2$ whenever $\mu$ is inside a gap with relative width at least~$\theta$, so, assuming that the a posteriori error bound or estimate is at least as accurate as the a priori one, we expect that all gaps of relative width $\theta$ or larger are detected by \cref{algorithm:eigenvalue-gap-finder--final}.
Numerical evidence shows that this assumption is generally satisfied when $\mu$ is not too close to the extrema of the gap and $\varepsilon$ is not too large.
We also mention that, although \cref{algorithm:eigenvalue-gap-finder--final} cannot explicitly find the gap between the $k$-th and the $(k+1)$-th eigenvalue for a given $k$, it still computes estimates for $\trace(P_{\mu_j})$ for all $\mu_j \in \vec \mu$, so we know the approximate number of eigenvalues below each detected gap and thus we can roughly determine if one of the gaps found by the algorithm is likely to be the one between $\lambda_k$ and~$\lambda_{k+1}$ (of course, we can only expect that \cref{algorithm:eigenvalue-gap-finder--final} will detect this gap if we assume that it is wide enough). If more accuracy on the number of eigenvalues below each gap is required, then one should use a larger number $s$ of sample vectors for Hutchinson's trace estimator. Alternatively, if it is feasible to compute an $LDL^T$ factorization, \rev{we can use a single factorization} to compute the exact number of eigenvalues below a certain gap in order to verify if it corresponds to the searched gap between the eigenvalues $\lambda_k$ and~$\lambda_{k+1}$. \rev{Observe that while computing such a factorization can be quite expensive, it would still be significantly cheaper than trying to find the gap between $\lambda_k$ and $\lambda_{k+1}$ by factorizing shifted matrices $A - \mu I$ in a trial-and-error approach, which would require the computation of several $LDL^T$ factorizations.}

\begin{remark}
	\label{rem:gapfinder-technical-details--horizontal-line}
	Recall that $\vec x^T P_\mu \vec x$ is a staircase-like piecewise constant function of $\mu$, and in particular it must be constant in any interval that contains no eigenvalues of $A$. This means that given a candidate gap $[\mu_1, \mu_2]$, in addition to the condition $\abs{\qfupperhat_\mathcal{M}(\mu_2) - \qflowerhat_\mathcal{M}(\mu_1)} \le \varepsilon$, we \rev{would} also need to check that $\qflowerhat_\mathcal{M}(\mu_2) < \qfupperhat_\mathcal{M}(\mu_1)$. If this additional condition is not satisfied, it is impossible for $\vec x^T P_\mu \vec x$ to be a constant function for $\mu \in [\mu_1, \mu_2]$, and hence we can conclude with certainty that there must be at least one eigenvalue of $A$ in the interval $[\mu_1, \mu_2]$. This seems to cause a contradiction when the condition $\abs{\qfupperhat_\mathcal{M}(\mu_2) - \qflowerhat_\mathcal{M}(\mu_1)} \le \varepsilon$ is verified but $\qflowerhat_\mathcal{M}(\mu_2) < \qfupperhat_\mathcal{M}(\mu_1)$ is not, since it would imply that all the eigenvalues of $A$ contained in~$[\mu_1, \mu_2]$ have an associated $\varepsilon$-small jump, which is a low-probability event. However, it simply means that this situation is very unlikely to happen in a practical scenario, and it would probably not be encountered if the algorithm was run on the same problem for a second time. \rev{In our implementation, we also check if the condition $\qflowerhat_{\mathcal{M}}(\mu_2) < \qfupperhat_{\mathcal{M}}(\mu_1)$ holds, even though in practice we never encountered a situation in which it was not satisfied.}
\end{remark}

\begin{remark}
	\label{rem:bisection-approach-mu}
	\rev{We mention that one could implement a variant of \cref{algorithm:eigenvalue-gap-finder--final} which does not require as input a vector of parameters $\vec \mu = [\mu_1, \dots, \mu_{N_f}]$, but instead detects gaps by sequentially evaluating the upper and lower bounds $\qfupperhat_\mathcal{M}(\mu)$ and~$\qflowerhat_\mathcal{M}(\mu)$ on suitably chosen points. Such an algorithm would likely employ fewer than $N_f$ values of $\mu$ to reach the same accuracy. However, since the bulk of the computational cost is in the construction of the Krylov subspace, the number of points $N_f$ has a negligible impact on performance (see \cref{subsec:simultaneous-computation-for-different-mu,subsec:computational-cost}) and thus we did not explore this approach in detail, opting instead for a more straightforward strategy with a sufficiently dense grid of points.}
\end{remark}

\subsection{Computational cost}
\label{subsec:computational-cost}

The most expensive operations in \cref{algorithm:eigenvalue-gap-finder--final} are the $m$ matrix-vector products with $A$ required to construct the Krylov basis in the Lanczos algorithm, which cost $O(m \, \nnz(A))$ for a sparse matrix $A$. The additional operations done for the computation of the quadratic forms $\vec x^T P_{\mu_j} \vec x$  with $N_f$ different parameters $\mu_j$ add up to $O(nm + m^2 + m N_f)$, as discussed in \cref{subsec:simultaneous-computation-for-different-mu}. 
If we estimate the error with the consecutive difference estimate, we have to run the Lanczos algorithm for one additional iteration, and compute the differences $\abs{\quadform_{m',j} - \quadform_{m'+1,j}}$ for all $m' \in \mathcal{M}$. If $\abs{\mathcal{M}} = d$ and the elements of $\mathcal{M}$ are consecutive integers, this amounts to computing $d+1$ quadratic forms $\quadform_{m', j}$ for all $j$, so the additional cost is $O(m^2 d + m N_f d)$.    

On the other hand, if we use the a posteriori error bound from \cref{prop:lanczos-a-posteriori-error-bound}, we have to evaluate the coefficients $\alpha_j$, $\beta_j$ and $\gamma_j$ for all $\mu_j$, which can be done with $O(m^2)$ operations assuming that the eigenvalue decomposition of $T_m$ is available, and compute the maximum of $\abs{g_m(z)}$ over $[\lambda_{\min}, \lambda_{\max}]$, where $g_m$ is defined in \cref{eqn:lanczos-post-bound--gm-definition}.
The evaluation of $g_m(z)$ for all~$\mu_j$ at a single point~$z$ requires $O(m N_f)$ operations, and if a discretization of $[\lambda_{\min}, \lambda_{\max}]$ with $L$ points is used, we can approximately evaluate the bound in \cref{prop:lanczos-a-posteriori-error-bound} for all $\mu_j$ with $O(m^2 + m N_f L)$ operations. 
If a set $\mathcal{M}$ with $d$ elements is used to compute the more robust bounds~$\qfupperhat_\mathcal{M}(\mu_j)$ and $\qflowerhat_\mathcal{M}(\mu_j)$, the cost for the computation of the bounds is roughly multiplied by a factor $d$ and becomes $O(m^2 d + m N_f L d)$, where we have also included the cost of computing the eigenvalues and eigenvectors of $T_{m'}$ for $m' \in \mathcal{M}$. 
Note that given $\qfupper_m(\mu_j)$ and~$\qflower_m(\mu_j)$, the refined bounds $\qfupperhat_m(\mu_j)$ and $\qflowerhat_m(\mu_j)$ described in \cref{subsec:detecting-gaps} can be computed for all~$\mu_j$ with~$O(N_f)$ operations. 
Finally, once $\qfupperhat_\mathcal{M}(\mu_j)$ and $\qflowerhat_\mathcal{M}(\mu_j)$ have been computed, the gaps can be located with a cost of~$O(N_f)$. 
The total computational cost of \cref{algorithm:eigenvalue-gap-finder--final} is therefore $O(m \, \nnz(A) + nm + m^2 d + m N_f d)$ with the consecutive difference estimate, and $O(m \, \nnz(A) + nm + m^2 d + m N_f L d)$ with the a posteriori bound from \cref{prop:lanczos-a-posteriori-error-bound}. 

Note that all operations associated with different values of $\mu$ are completely independent, so the algorithm can be easily parallelized, for instance by assigning different slices of the spectral interval to different processors. This parallelization is not particularly helpful when the matrix size $n$ is large enough, since the dominant part of the computation would still be the $O(m \, \nnz(A))$ term associated with matrix-vector products with $A$; however, parallelizing the computations associated with different parameters $\mu$ can be helpful for mitigating the $O(m N_f L d)$ cost of computing the a posteriori error bounds from \cref{prop:lanczos-a-posteriori-error-bound}, which can become quite high if both $N_f$ and $L$ are large, see, for instance, the experiment in \cref{subsec:experiments-real-world}.   

If our goal is to detect all gaps with relative width larger than $\theta$ with probability at least $1 - \delta$, we should take $m$ according to \cref{eqn:lanczos-iterations-m-best-choice} with $\varepsilon \le e^{-1} \delta^2$, i.e.,
\begin{equation}
	\label{eqn:lanczos-iterations-m--best-choice--delta-theta}
	m \ge 1 + \frac{1 + \log\big(2 C_\rev{\theta} \norm{\vec x}_2^2 / \delta^2 \big)}{\log\big((1+\theta)/(1-\theta)\big)}, \qquad C_\rev{\theta} = \frac{1 - \theta}{\sqrt{\pi \theta}} + 1,
\end{equation} 
where we recall that $\expec[\norm{\vec x}_2^2] = n$.  
By looking at the behavior of the right-hand side of \cref{eqn:lanczos-iterations-m--best-choice--delta-theta} as~$\theta \to 0$, using the fact that 
\begin{equation*}
	\log \bigg( \frac{1+\theta}{1-\theta} \bigg) \approx \frac{1}{2\theta},
\end{equation*} 
we conclude that for small $\theta$ we should take
\begin{equation*}
	m \approx \log \bigg(\frac{2 \norm{\vec x}_2^2}{\delta^2 \sqrt{\pi \theta}}\bigg) \frac{1}{2 \theta}.
\end{equation*}
Although the dependence of $m$ on the failure probability $\delta$ and the matrix size $n$ is not particularly severe since they appear inside a logarithm, the main drawback of this approach is the $O(1/\theta)$ dependence on the relative gap width $\theta$, which causes \cref{algorithm:eigenvalue-gap-finder--final} to quickly become more expensive as the target gap width is reduced. This is an inherent limitation of the Lanczos method for the computation of $\vec x^T P_\mu \vec x$, and more precisely of the polynomial approximation of the sign function (see \cref{prop:sign-polynomial-approx-bound--hasson}). As we mentioned in \cref{subsec:algorithm-analysis--lanczos}, \cref{prop:sign-polynomial-approx-bound--hasson} can be improved by considering two nonsymmetric intervals as in \cite{EremenkoYuditskii11}, and this would in turn lead to an improved estimate on the required number of Lanczos iterations $m$; however, in general we do not expect to find an asymptotic estimate that is better than $m = O(1/\theta)$. Indeed, for a gap located near the center of $\spec(A)$, the spectrum of $A$ is contained in two intervals that are approximately symmetric with respect to the gap, and there would not be any significant advantage in using the improved result \cite[Theorem~1.1]{EremenkoYuditskii11}.

\begin{example}
	\label{example:trace-bound-final}

In this example we show how the bounds described in \cref{subsec:detecting-gaps} perform on the simple test problem from \cref{example:hutchinson-projector-traceest}. The matrix $A$ is $600 \times 600$, with $\spec(A) \subset [0, 60]$ and three gaps in the spectrum given by the intervals $[20, 21]$, $[30, 32]$ and $[40, 44]$. We use a single Gaussian vector $\vec x \in \R^n$ for Hutchinson's estimator and $m = 50$ Lanczos iterations. We compare the exact $\vec x^T P_\mu \vec x$ with the upper and lower bounds $\qfupperhat_\mathcal{M}(\mu)$ and $\qflowerhat_\mathcal{M}(\mu)$ obtained both with \cref{prop:lanczos-a-posteriori-error-bound} and with the consecutive difference error estimate. In particular, for the latter we use the error estimate $\postbound_m(\mu_j) = c \abs{\quadform_{m,j} - \quadform_{m+1, j}}$, with a safety factor $c = 2$. In both cases, we take $\mathcal{M} = \{ m-d+1, \dots, m \}$ with $d = 3$. 
Since $c \abs{\quadform_{m,j} - \quadform_{m+1,j}}$ is not guaranteed to be an upper bound for $\abs{\vec x^T P_{\mu_j} \vec x}$, the inequalities $\qflowerhat_\mathcal{M}(\mu) \le \vec x^T P_\mu \vec x \le \qfupperhat_\mathcal{M}(\mu)$ do not necessarily hold for all $\mu$, especially after the two minimizations and maximizations in \cref{eqn:lanczos-post-bound-monotone} and \cref{eqn:lanczos-upper-lower-a-posteriori-bounds}. To overcome this issue, we also propose to use a ``safe'' variant of the upper and lower estimates for $\vec x^T P_\mu \vec x$, defined as
\begin{equation}
	\label{eqn:lanczos-upper-lower-robust-safe-estimates}
	\qfupperhatsafe_\mathcal{M}(\mu) := \max_{m \in \mathcal{M}}\qfupperhat_m(\mu) \qquad \text{and} \qquad \qflowerhatsafe_\mathcal{M}(\mu) := \min_{m \in \mathcal{M}}\qflowerhat_m(\mu),
\end{equation} 
which correspond to taking the worst of the bounds associated with $\mathcal{M}$ instead of the best ones. This variant significantly increases the chance that $\qfupperhatsafe_\mathcal{M}(\mu)$ and $\qflowerhatsafe_\mathcal{M}(\mu)$ are an upper and lower bound for~$\vec x^T P_\mu \vec x$, respectively, although it reduces the overall accuracy of the estimates.

The bounds are shown in \cref{fig:example-trace-bound-final}, including the safe variants $\qfupperhatsafe_\mathcal{M}(\mu)$ and $\qflowerhatsafe_\mathcal{M}(\mu)$ of the bounds obtained by using consecutive differences.
Notice that the upper and lower estimates $\qfupperhat_\mathcal{M}(\mu)$ and $\qflowerhat_\mathcal{M}(\mu)$ computed with consecutive differences are much closer to $\vec x^T P_\mu \vec x$ compared to the bounds obtained with \cref{prop:lanczos-a-posteriori-error-bound}, but they sometimes fail to satisfy $\qflowerhat_\mathcal{M}(\mu) \le \vec x^T P_\mu \vec x \le \qfupperhat_\mathcal{M}(\mu)$: see for instance the right plot in \cref{fig:example-trace-bound-final}, which is a close-up of the larger gap $[40, 44]$. On the other hand, the safe variants $\qfupperhatsafe_\mathcal{M}(\mu)$ and $\qflowerhatsafe_\mathcal{M}(\mu)$ are more robust, while still being more accurate than the a posteriori bounds from \cref{prop:lanczos-a-posteriori-error-bound}. To provide a quantitative comparison, for the plot in \cref{fig:example-trace-bound-final} we used $4000$ different shift parameters~$\mu_j$, and the inequality $\vec x^T P_{\mu_j} \vec x \le \qfupperhat_\mathcal{M}(\mu_j)$ was not satisfied for $251$ of them, while $\vec x^T P_{\mu_j} \vec x \le \qfupperhatsafe_\mathcal{M}(\mu_j)$ was not satisfied for only $35$ values of $j$. Note that the fraction of values of $\mu$ for which the safe upper estimate $\qfupperhatsafe_{\mathcal{M}}(\mu)$ fails to be an upper bound is smaller than $10^{-2}$, which is the failure probability that we use for the experiments in \cref{sec:numerical-experiments}.

\begin{figure}[t]
	\centering
	\makebox[\linewidth][c]{
		\begin{subfigure}[t]{.50\textwidth}
			\includegraphics[width=\textwidth]{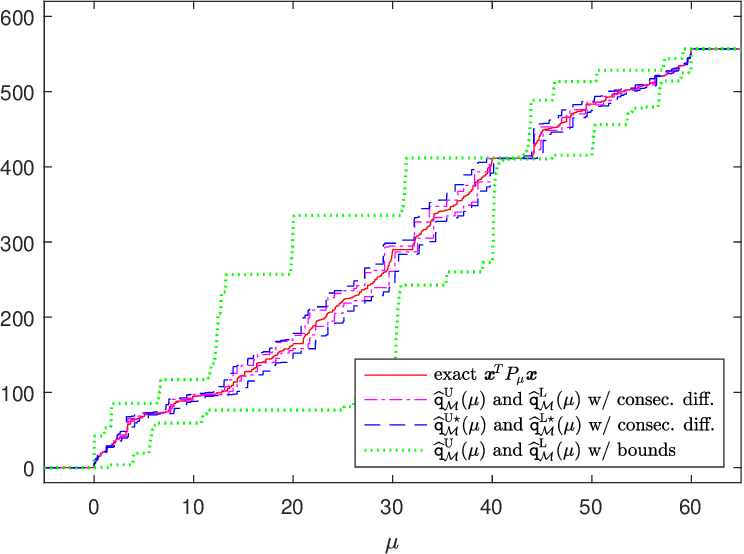}
		\end{subfigure}
		\begin{subfigure}[t]{.50\textwidth}
			\includegraphics[width=\textwidth]{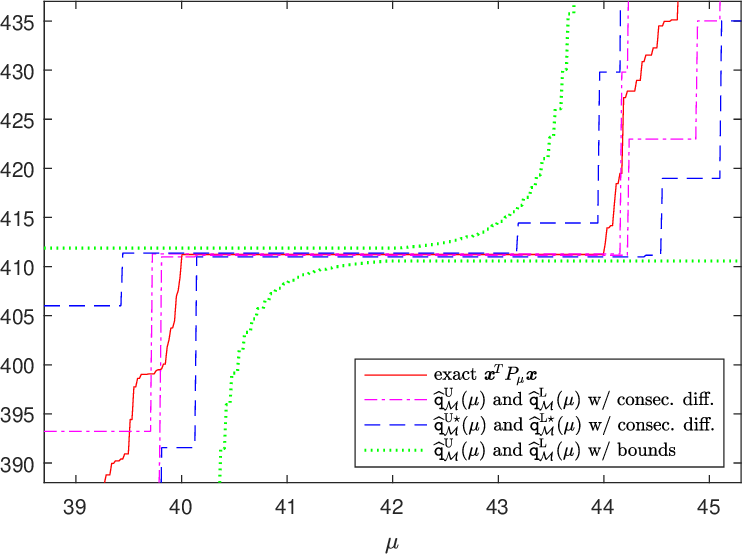}
		\end{subfigure}
		}
	\caption[Upper and lower bounds for $\vec x^T P_\mu \vec x$ for several $\mu$]{Upper and lower bounds for $\vec x^T P_\mu \vec x$ obtained with the Lanczos algorithm with $m = 50$ iterations. The right plot is a close-up of the left plot around the larger gap $[40, 44]$.  
	\label{fig:example-trace-bound-final}}
\end{figure}

\end{example}

\rev{

\subsection{Literature comparison}	
\label{subsec:literature-comparison}

As we already mentioned in some of the previous sections, the combination of Hutchinson's trace estimator with the Lanczos algorithm has been used in several works on related problems, for instance in \cite{CTU21,BKM22,DPS16}. The main aspect that sets our work apart from the rest of the literature is the focus on the detection of gaps, and the consequent difference in the theoretical analysis that we perform. 
Indeed, most results in the literature on spectral density approximation are concerned with global approximation of the spectral density, such as in the Wasserstein distance \cite[Theorem~1]{CTU21}, so they are unable to exploit the faster convergence of the Lanczos method for $\vec x^T h_\mu(A) \vec x$ when $\mu$ is located inside a large gap (see \cref{prop:lanczos-matfun-approx--apriori-bound}). Moreover, convergence results found in the literature are usually based on the approximation error for $\trace(P_\mu)$ with Hutchinson's estimator, using results such as \cref{prop:hutchinson-tail-bound--gaussian}. In our context, we require an absolute error on $\eigcount(\mu) = \trace(P_\mu)$ smaller than $\frac{1}{2}$, which corresponds to a relative error of $\frac{1}{2n}$ on the normalized spectral density; as we observed at the end of \cref{subsec:choosing-input-parameters}, achieving this level of accuracy with a bound such as \cite[Theorem~1]{CTU21} would require us to take $s = O(n)$ vectors for Hutchinson's estimator, which are too many to make this approach feasible. 

On the other hand, our analysis based on the height of the jumps in $\vec x^T h_\mu(A) \vec x$ and $\traceH{s}(P_\mu)$  bypasses the need to obtain a highly accurate approximation of the trace and allows us to use significantly fewer matrix-vector products. The main tradeoff is that we do not obtain the precise number of eigenvalues below each detected gap, but only an estimate.
To the best of our knowledge, this perspective and the resulting approach for the theoretical analysis have not appeared before in the literature.
Although we have not explored this direction, it is possible that some of the techniques used in this work may also be applicable to related research areas in which the combination of Hutchinson and Lanczos is employed, such as spectral density estimation.

}

\section{Numerical experiments}
\label{sec:numerical-experiments}

In this section we run some experiments to investigate the performance of \cref{algorithm:eigenvalue-gap-finder--final}. 
For all the experiments, we use \cref{algorithm:eigenvalue-gap-finder--final} with a failure probability $\delta = 10^{-2}$ and a number of Lanczos iterations~$m$ that is either fixed in advance or computed according to \cref{eqn:lanczos-iterations-m--best-choice--delta-theta} in order to find all gaps with relative width larger than a certain target $\theta$. When not stated differently, for the Lanczos algorithm we employ the error estimate $\postbound_m(\mu_j) = c \abs{\quadform_{m, j} - \quadform_{m, j+1}}$ with a safety factor $c = 2$, and we use the more robust upper and lower estimates $\qfupperhatsafe(\mu)$ and $\qflowerhatsafe(\mu)$ defined in \cref{eqn:lanczos-upper-lower-robust-safe-estimates}, with 
$\mathcal{M} = \{m-d+1, \dots, m\}$ with $d = 3$. These estimates are cheaper to compute compared to the a posteriori error bound from \cref{prop:lanczos-a-posteriori-error-bound}, and they are also quite reliable, as shown in \cref{example:trace-bound-final}. Even though the estimate $c\abs{\quadform_{j,m} - \quadform_{j, m+1}}$ does not have the same theoretical guarantees as the a posteriori error bound, we have seen in \cref{example:lanczos-bound-comparison} that it is usually much closer to the actual error. All the experiments have been done using MATLAB \rev{R2024b} on a laptop running Ubuntu \rev{24.04}, with 32 GB of RAM and an Intel Core \rev{Ultra 9 185H} CPU with base clock rate \rev{2.3 GHz}. 
The code to reproduce the experiments in this section is available on Github at \url{https://github.com/simunec/eigenvalue-gap-finder}.

\subsection{Scaling with the gap width}
\label{experiments-gap-width}
We start by examining the behavior of \cref{algorithm:eigenvalue-gap-finder--final} with respect to the relative gap width $\theta$.

We consider a symmetric tridiagonal matrix $A$ of size $n = 30000$ obtained as a sum $A = D + T$, where~$T$ is a random symmetric tridiagonal matrix with $\mathcal{N}(0,1)$ entries and $D$ is a diagonal matrix. The eigenvalues of $D$ are chosen so that the spectrum of $A$ has a gap with relative width approximately $\theta$. Specifically, the eigenvalues of $D$ are contained in the intervals $[1, 10^3] \cup [10^3 + x, 10^4]$, where $x$ satisfies
\begin{equation*}
	\frac{x/2}{10^4 - 10^3 - x/2} = \theta,
\end{equation*}
so that the relative width of the gap $[10^3, 10^3+x]$ is $\theta$, and hence the spectrum of the perturbed matrix $A = D + T$ also has a gap with relative width approximately $\theta$.
We consider a matrix $D$ with $20000$ logspaced eigenvalues in $[1, 10^3]$ and $10000$ logspaced eigenvalues in $[10^3+x, 10^4]$.

We estimate this gap using \cref{algorithm:eigenvalue-gap-finder--final}, with $m$ computed according to \cref{eqn:lanczos-iterations-m--best-choice--delta-theta}, so that we can expect it to successfully find all gaps with relative width larger than $\theta$; we use $N_f = 10000$ parameters $\mu$, logarithmically spaced on the interval $[1, 10^4]$. 
For each $\theta$, in \cref{table:gapfinder-gapscale} we show the true gap computed by diagonalizing $A$, the gap estimated by \cref{algorithm:eigenvalue-gap-finder--final} and the approximate number of eigenvalues below the gap, obtained by rounding the trace approximation $\vec x^T P_\mu \vec x$ for $\mu$ equal to the left endpoint of the estimated gap, as well as the number of Lanczos iterations $m$, the execution time of the algorithm and the time required for the computation of the eigenvalues of $A$. 
We see that the gap is found successfully for all $\theta$, with the estimated gap being slightly larger than the true gap only for $\theta = 0.1$ and $\theta = 0.05$. The estimated number of eigenvalues below the gap roughly approximates the exact value $20000$, but it is still relatively far from the correct number since we are using only one sample vector for Hutchinson's estimator. 
Note that the number of Lanczos iterations scales approximately as~$O(1/\theta)$, but the execution time increases more than linearly in the number of Lanczos iterations $m$; this is likely due to the $O(m^2)$ cost for the computation of eigenvalues and eigenvectors of the projected matrix $T_m$ becoming \rev{an increasingly} dominant part of the computation when the subspace dimension~$m$ is large. On the other hand, the execution time for $\texttt{eig}$ remains constant, so when the gap width $\theta$ becomes sufficiently small it would be more efficient to directly diagonalize the matrix instead of using \cref{algorithm:eigenvalue-gap-finder--final}. Of course, this would not be the case if the matrix dimension $n$ increases, as we show in the following experiment.

\begin{table}
	\caption[Performance of \cref{algorithm:eigenvalue-gap-finder--final} with decreasing gap width]{Performance of \cref{algorithm:eigenvalue-gap-finder--final} with $n = 30000$ and varying gap width $\theta$. }
	\label{table:gapfinder-gapscale}
	\begin{tabular}{S[table-format=1.4]|ccS[table-format=5.0]S[table-format=4.0]cc}
	\toprule
	$\theta$  & true gap & est.~gap & {est.~$\eigcount(\mu)$} & {$m$} & {time (s)} & {time \texttt{eig} (s)}\\
	\midrule
	0.1 & $[$1001.80, 2633.87$]$ & $[$1016.48, 2637.19$]$ & 19637 & 112 & \rev{0.103} & \rev{2.318} \\
	0.05 & $[$1001.65, 1856.42$]$ & $[$999.77, 1856.64$]$ & 19776 & 226 & \rev{0.174} & \rev{2.212} \\
	0.025 & $[$1001.15, 1435.81$]$ & $[$1001.61, 1434.55$]$ & 20198 & 456 & \rev{0.312} & \rev{2.293} \\
	0.01 & $[$1001.17, 1175.30$]$ & $[$1001.61, 1174.65$]$ & 20027 & 1156 & \rev{0.789} & \rev{2.077} \\
	0.005 & $[$1000.69, 1085.83$]$ & $[$1000.69, 1085.18$]$ & 19862 & 2342 & \rev{1.894} & \rev{2.308} \\
	0.0025 & $[$1003.69, 1042.75$]$ & $[$1004.38, 1042.08$]$ & 19709 & 4745 & \rev{5.537} & \rev{2.311} \\
	\bottomrule
\end{tabular}
\end{table}

\subsection{Scaling with the matrix size}
\label{subsec:experiments-matrix-size}

Let us now consider a sequence of matrices with increasing size and approximately constant gap width. We use a setup similar to \cref{subsec:experiments-matrix-size}, with fixed relative gap width $\theta = 0.01$ and increasing matrix size~$n$. We take $A = D + T$, where $T$ is symmetric tridiagonal with random $\mathcal{N}(0,1)$ entries and~$D$ is diagonal with $\spec(A) \subset [1, 10^3] \cup [10^3 + x, 10^4]$, with~$n/2$ logspaced eigenvalues below the gap and~$n/2$ logspaced eigenvalues above the gap. The value of $x$ is chosen so that $A$ has relative gap width approximately $\theta$, as in \cref{subsec:experiments-matrix-size}.

We use \cref{algorithm:eigenvalue-gap-finder--final} to estimate the gap, with the number of Lanczos iterations $m$ computed according to \cref{eqn:lanczos-iterations-m--best-choice--delta-theta} and $N_f = 10000$ parameters $\mu$, logarithmically spaced over the interval $[1, 10^4]$. The results for a matrix dimension that increases from $n = 5000$ to $n = 80000$ are shown in \cref{table:gapfinder-sizescale}. In this case, the gap is found correctly for all matrix sizes, with none of the estimated gaps being larger than the true gap computed via a diagonalization. Note that the number of Lanczos iterations increases very slowly with the matrix dimension, in accordance with the discussion in \cref{subsec:computational-cost}. Hence, the execution time of the algorithm increases roughly linearly, since with a large matrix size~$n$ and moderate subspace dimension~$m$ the leading term in the computational cost is given by the $m$ matrix-vector products with $A$ in the construction of the Lanczos basis, which cost $O(mn)$ for a tridiagonal matrix. On the other hand, the time required for computing the eigenvalues of the tridiagonal matrix~$A$ scales like $O(n^2)$, so it quickly becomes more expensive than \cref{algorithm:eigenvalue-gap-finder--final}. Note that if $A$ was a sparse matrix that is not tridiagonal, the diagonalization time would scale like $O(n^3)$, making the difference between the two approaches even more evident.

It is interesting to observe that \cref{algorithm:eigenvalue-gap-finder--final} can be also used to cheaply obtain a rough estimate of the gap by taking a larger failure probability $\delta$ and fixing the number of Lanczos iterations $m$ in advance instead of computing it with \cref{eqn:lanczos-iterations-m--best-choice--delta-theta}. For instance, we show in \cref{table:gapfinder-sizescale-rough} the performance of \cref{algorithm:eigenvalue-gap-finder--final} with $\delta = 0.5$ and $m = 250$. We see that the gap in the spectrum is still located successfully, although the accuracy is generally lower and sometimes the estimated gap is larger than the true gap, for instance in the case $n = 5000$. Note that the estimated number of eigenvalues below the gap coincides with the corresponding number from \cref{table:gapfinder-sizescale}, even if the number of Lanczos iterations is different: this happens because we used the same random vector $\vec x$ for Hutchinson's trace estimator, and in both cases we are approximating the same quadratic form $\vec x^T P_\mu \vec x$ (the values of $\mu$ in the two cases may be different, but as long as they are both inside the gap the corresponding values of~$\vec x^T P_\mu \vec x$ coincide). In particular, this also confirms that the error in the estimated number of eigenvalues below the gap is essentially only due to the error in the stochastic trace estimator, and not to the error in the approximation of the quadratic form via the Lanczos algorithm.

\begin{table}
	\centering
	\caption[Performance of \cref{algorithm:eigenvalue-gap-finder--final} with increasing matrix size]{Performance of \cref{algorithm:eigenvalue-gap-finder--final} with fixed gap width $\theta = 0.01$ and increasing matrix size $n$.}
	\label{table:gapfinder-sizescale}
	\begin{tabular}{S[table-format=5.0]|ccS[table-format=5.0]S[table-format=4.0]cS[table-format=2.3]}
		\toprule
		$n$  & true gap & est.~gap & {est.~$\eigcount(\mu)$} & {$m$} & {time (s)} & {time \texttt{eig} (s)}\\
		\midrule
		5000 & $[$1000.58, 1177.15$]$ & $[$1000.69, 1176.81$]$ & 2377 & 1067 & \rev{0.259} & \rev{0.077} \\
		10000 & $[$1001.11, 1175.88$]$ & $[$1001.61, 1175.73$]$ & 4956 & 1101 & \rev{0.405} & \rev{0.291} \\
		20000 & $[$1000.23, 1177.01$]$ & $[$1000.69, 1176.81$]$ & 9835 & 1136 & \rev{0.572} & \rev{1.071} \\
		40000 & $[$1002.67, 1175.70$]$ & $[$1003.46, 1174.65$]$ & 19930 & 1171 & \rev{0.984} & \rev{3.696} \\
		80000 & $[$1001.39, 1176.18$]$ & $[$1001.61, 1175.73$]$ & 39874 & 1205 & \rev{1.842} & \rev{13.161} \\
		\bottomrule
	\end{tabular}
\end{table}

\begin{table}
	\centering
	\caption[Performance of \cref{algorithm:eigenvalue-gap-finder--final} with fixed number of Lanczos iterations]{Performance of \cref{algorithm:eigenvalue-gap-finder--final} with fixed gap width $\theta = 0.01$ and increasing matrix size $n$. The number of Lanczos iterations is fixed to $m = 250$ and the failure probability is set to $\delta = 0.5$.}
	\label{table:gapfinder-sizescale-rough}
	\begin{tabular}{S[table-format=5.0]|ccS[table-format=5.0]S[table-format=3.0]cS[table-format=2.3]}
		\toprule
		$n$  & true gap & est.~gap & {est.~$\eigcount(\mu)$} & {$m$} & {time (s)} & {time \texttt{eig} (s)}\\
		\midrule
		5000 & $[$1000.58, 1177.15$]$ & $[$998.85, 1178.98$]$ & 2377 & 250 & \rev{0.072} & \rev{0.077} \\
		10000 & $[$1001.11, 1175.88$]$ & $[$1003.46, 1170.33$]$ & 4956 & 250 & \rev{0.098} & \rev{0.287} \\
		20000 & $[$1000.23, 1177.01$]$ & $[$1001.61, 1167.10$]$ & 9835 & 250 & \rev{0.145} & \rev{1.037} \\
		40000 & $[$1002.67, 1175.70$]$ & $[$1002.54, 1161.73$]$ & 19930 & 250 & \rev{0.230} & \rev{3.748} \\
		80000 & $[$1001.39, 1176.18$]$ & $[$1005.31, 1175.73$]$ & 39874 & 250 & \rev{0.399} & \rev{13.085} \\
				\bottomrule
	\end{tabular}
\end{table}

\subsection{Examples from applications}
\label{subsec:experiments-real-world}

\rev{In this section, we demonstrate the effectiveness of \cref{algorithm:eigenvalue-gap-finder--final} on two illustrative examples arising from practical applications.}

\subsubsection{Dirac comb Hamiltonian}
\label{subsubsec:experiments-diraccomb}

\rev{
As a first model problem, we consider the one-dimensional discrete Hamiltonian operator
\begin{equation*}
	A = -L + V,
\end{equation*}
where $L$ is a finite difference discretization of the second derivative on $[0,N]$ with periodic boundary conditions using a uniform grid with $k$ points on each unit interval, for a total of $n = kN$ points, and $V$ is a diagonal matrix representing a discrete \textit{Dirac comb} periodic potential, which takes the value $k^2$ at the points $x = 1, 2, \dots, N-1$, and is zero otherwise. This problem is a toy model that mimics the properties of a Hamiltonian of a one-dimensional physical system with a periodic potential, whose spectrum is typically characterized by several gaps, representing forbidden energy levels for the electrons in the system; see for instance \cite[Chapter~7]{KM18} and \cite{LHZ25} for more details.

We use a discretization with $N = 2000$ and $k = 5$ for this example, for a total of $n = 10000$ discretization points. The spectrum of the resulting symmetric tridiagonal matrix $A \in \R^{n \times n}$ has five bands, separated by four relatively large gaps, each containing a single isolated eigenvalue, which splits it into two parts; we will refer to them respectively as the lower and upper part of the gap. We shift and scale $A$ so that its spectrum is contained in the interval $[0, 10]$, to facilitate visualization of the gaps.  
We run \cref{algorithm:eigenvalue-gap-finder--final} to detect the gaps, both with the robust consecutive difference estimates \cref{eqn:lanczos-upper-lower-robust-safe-estimates} and with the error bounds from \cref{prop:lanczos-a-posteriori-error-bound}, using $m = 150$ Lanczos iterations and $d = 3$. In \cref{fig:gapfinder-diraccomb-traceest} we show the spectrum of $A$, the approximations to $\vec x^T P_\mu \vec x$, and the upper and lower bounds obtained with the two approaches, where $\vec x$ is the single random vector used in Hutchinson's estimator; we use the same vector $\vec x$ for the two variants of \cref{algorithm:eigenvalue-gap-finder--final}. 
Note that the bounds from \cref{prop:lanczos-a-posteriori-error-bound} are significantly less accurate than the estimates \cref{eqn:lanczos-upper-lower-robust-safe-estimates} in regions far from the gaps, but both approaches are extremely accurate inside the gaps. 
with the chosen parameters, \cref{algorithm:eigenvalue-gap-finder--final} successfully detects both parts of the first gap, the lower half of the second and third gaps, and both parts of the fourth gap; by increasing the number of Lanczos iterations to $m = 250$, all gaps are successfully detected.

In \cref{table:gapfinder-diraccomb} we compare the exact gaps obtained by directly computing the eigenvalues of $A$ with the approximate ones obtained with \cref{algorithm:eigenvalue-gap-finder--final}, for the lower half of the four gaps, along with the time required to run each method. 
Note that the gaps found by using the a posteriori bounds from \cref{prop:lanczos-a-posteriori-error-bound} are slightly smaller than the ones obtained by estimating the Lanczos error with consecutive differences; this is a consequence of the lower accuracy of the a posteriori bounds. Moreover, the execution time is quite higher with the a posteriori bounds, due to the $O(m N_f L d)$ term in the computation of the error bounds from \cref{prop:lanczos-a-posteriori-error-bound}; indeed, it turns out that this is the most expensive part of the computation in this setting, where we use $N_f = 1000$ different values of $\mu$ and a discretization with $L = 1000$ points to evaluate the maximum of $\abs{g_m(z)}$ over the spectral interval of~$A$. On the other hand, when the error in the Lanczos algorithm is estimated using consecutive differences, the portion of the computational cost that depends on $N_f$ scales as $O(m N_f d)$, which is significantly less expensive than the error bounds from \cref{prop:lanczos-a-posteriori-error-bound} when $L = 1000$ (see \cref{subsec:computational-cost}).

}

\begin{table}
	\centering
	\caption[Performance of \cref{algorithm:eigenvalue-gap-finder--final} on the discrete Hamiltonian with Dirac comb potential]{Performance of \cref{algorithm:eigenvalue-gap-finder--final} on the discrete Hamiltonian with a Dirac comb potential, using the robust consecutive difference estimate \cref{eqn:lanczos-upper-lower-robust-safe-estimates} and the a posteriori bound from \cref{prop:lanczos-a-posteriori-error-bound}. The number of Lanczos iterations is set to $m = 150$ and the failure probability to $\delta = 10^{-2}$. The gaps in the table correspond to the lower parts of the gaps in \cref{fig:dirac-comb-spectrum}.}
	\label{table:gapfinder-diraccomb}
	\begin{tabular}{l|ccccS[table-format=2.3]}
		\toprule
		method  & first gap & second gap & third gap & fourth gap & {time (s)} \\
		\midrule
		\texttt{eig} & [0.739, 1.326] & [3.101, 3.828] & [6.021, 6.757] & [8.383, 8.937] & 9.498 \\
		consec.~diff. & [0.741, 1.321] & [3.103, 3.824] & [6.106, 6.747] & [8.418, 8.929] & 0.028 \\
		\cref{prop:lanczos-a-posteriori-error-bound} & [0.781, 1.321] & [3.223, 3.754] & [6.136, 6.717] & [8.418, 8.929] & 0.365 \\
				\bottomrule
	\end{tabular}
\end{table}

\begin{figure}
	\centering
	\includegraphics[width=0.45\textwidth]{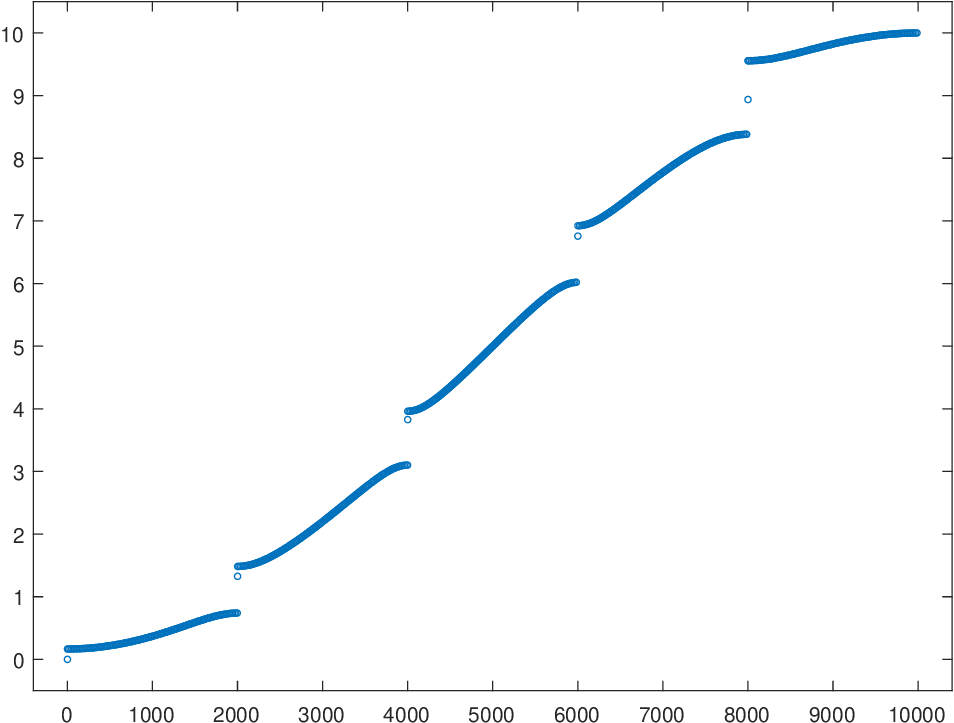}
	\caption[Spectrum of the Dirac comb Hamiltonian]{Spectrum of the Dirac comb discrete Hamiltonian.   
	\label{fig:dirac-comb-spectrum}}
\end{figure}

\begin{figure}
	\centering
	\makebox[\linewidth][c]{
		\begin{subfigure}[t]{.5\textwidth}
			\includegraphics[width=\textwidth]{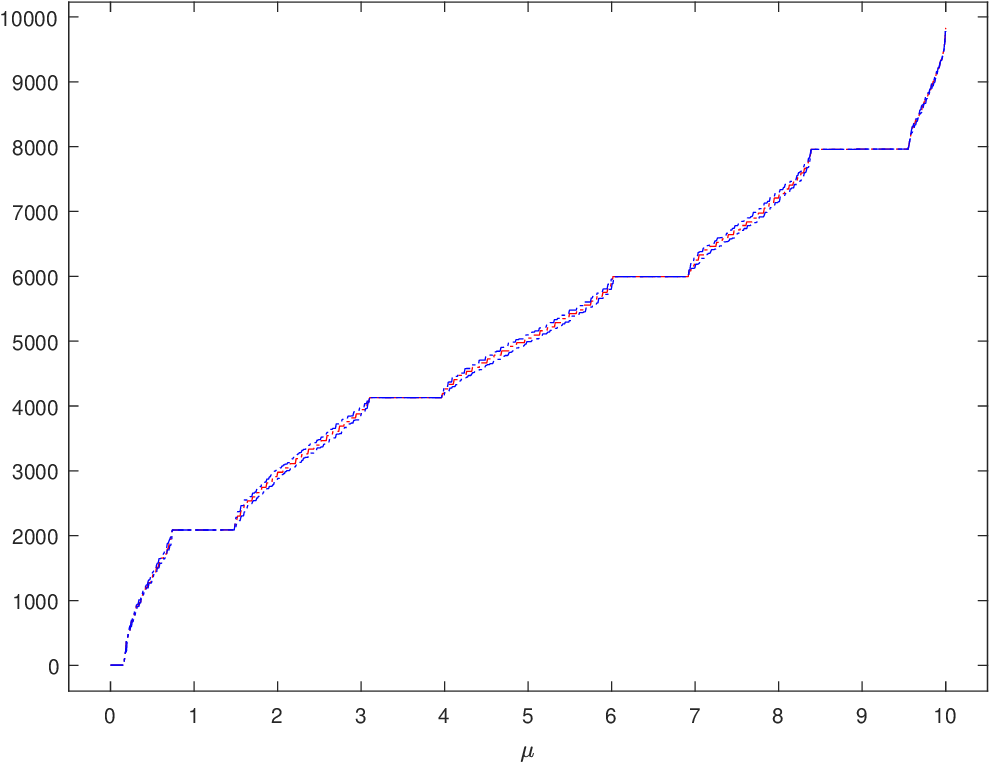}
		\end{subfigure}
		\begin{subfigure}[t]{.5\textwidth}
			\includegraphics[width=\textwidth]{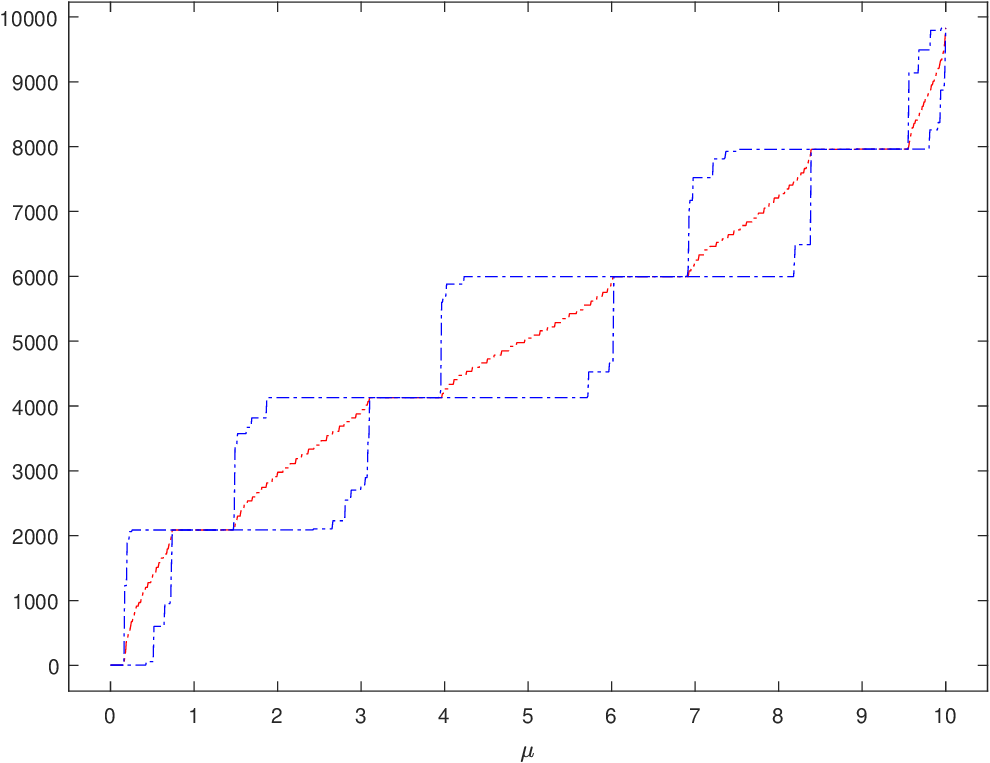}
		\end{subfigure}
		}
	\caption[Approximation of $\trace(P_\mu)$ with Hutchinson's estimator for several $\mu$]{Approximations to $\vec x^T P_\mu \vec x$, and upper and lower bounds obtained with \cref{algorithm:eigenvalue-gap-finder--final} for different $\mu$, with $m = 150$ Lanczos iterations, for the problem in \cref{subsubsec:experiments-diraccomb}. Left: \cref{algorithm:eigenvalue-gap-finder--final} with robust consecutive difference error estimate \cref{eqn:lanczos-upper-lower-robust-safe-estimates}. Right: \cref{algorithm:eigenvalue-gap-finder--final} with a posteriori error bound from \cref{prop:lanczos-a-posteriori-error-bound}.
	\label{fig:gapfinder-diraccomb-traceest}}
\end{figure}

\subsubsection{Hydrogen molecule chain}
\label{subsubsec:experiments-hydrogen}

\rev{
As an example from electronic structure computations, we consider the discretized Hamiltonian associated with a linear chain of 1250 H$_2$ molecules. In contrast to the atomic hydrogen chain, which is widely used as a theoretical benchmark but does not occur in nature, the molecular chain corresponds to a physically realizable system while still being straightforward to define; see, e.g., \cite{BiborskiH2chain}.

The matrices are generated with the open-source package PySCF~\cite{PySCF}, a Python package for performing electronic structure computations using Density Functional Theory. The code returns the Hamiltonian $\widetilde{A}$, obtained through a Galerkin discretization over a basis of Gaussian functions, and the overlap matrix $S$. Although the basis functions have global support, the Hamiltonian and overlap matrices are approximately sparse: most entries decay rapidly with the distance from the main diagonal and are set to zero once they fall below the threshold used for the integrals\footnote{In libcint, the constant \texttt{EXPCUTOFF} is set to $10^{-15}$ 
in \texttt{cint.h}; see the libcint GitHub repository \url{https://github.com/sunqm/libcint}.}.
The basis can be orthogonalized by taking the Cholesky factorization $S=R^T R$, where $S$ is the overlap matrix. Thus, the new Hamiltonian is $A = R^{-T}\widetilde{A}R^{-1}$, and the energy levels correspond to the eigenvalues of $A$; see~\cite[Section~3]{BenziProjector13}. When the intermolecular spacing in the chain is sufficiently large, the system behaves as a semiconductor~\cite{BiborskiH2chain}, i.e., the spectrum exhibits a gap between occupied and unoccupied states. To ensure this, we impose an intermolecular distance of $2.5$\AA, while the bond length within the H$_2$ molecule is set to $0.74$\AA.

The test matrix $A$ has size $n=5000$. The spectrum, shown in \cref{fig:gapfinder-h2-spectrum}, has two relatively large gaps and a smaller gap in the middle of the spectrum. There is a single isolated eigenvalue within the third gap, splitting it into two parts. In electronic structure computations, the gap of interest is the HOMO-LUMO gap, which in this setting corresponds to the first gap in \cref{fig:gapfinder-h2-spectrum}. The original Hamiltonian $\widetilde{A}$ and the Cholesky factor $R$ have 189622 and 99804 nonzero entries, respectively. Although $A$ is also approximately sparse, in order to compute the matrix-vector products needed for the construction of the Krylov basis it is more computationally efficient to use the representation $A=R^{-T}\widetilde A R^{-1}$ instead of forming $A$ explicitly, so we rely on the solution of the linear systems with $R$ and $R^T$, which are triangular and sparse, and on multiplications with $\widetilde{A}$.

In this test we use a setup similar to \cref{subsubsec:experiments-diraccomb}. We find the gaps in the spectrum of $A$ using \cref{algorithm:eigenvalue-gap-finder--final}, and we compare the robust consecutive difference estimates defined in \cref{eqn:lanczos-upper-lower-robust-safe-estimates} with the a posteriori bound from \cref{prop:lanczos-a-posteriori-error-bound}, both obtained using $m = 100$ Lanczos iterations and $d = 3$; we use $N_f = 1000$ values of $\mu$, and a discretization with $L = 1000$ points for evaluating the a posteriori error bounds.
We plot in \cref{fig:gapfinder-h2-traceest} the approximations to $\vec x^T P_\mu \vec x$ and the upper and lower bounds obtained with the two approaches. We again observe that the upper and lower bounds obtained by using \cref{prop:lanczos-a-posteriori-error-bound} are much more conservative compared to those based on consecutive difference estimates, especially far from the gaps in the spectrum.
}

\begin{figure}
	\centering
	\includegraphics[width=0.5\textwidth]{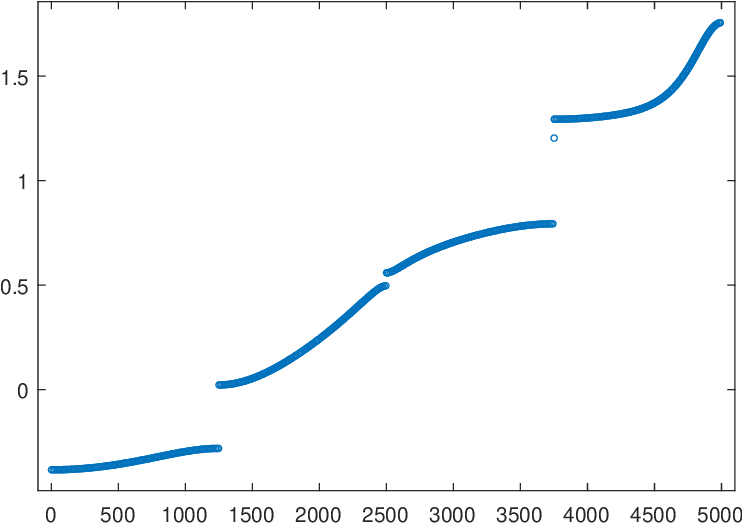}
	\caption[Spectrum of the discrete Hamiltonian]{Spectrum of the discrete Hamiltonian for the hydrogen molecule chain.   
	\label{fig:gapfinder-h2-spectrum}}
\end{figure}

\begin{figure}
	\centering
	\makebox[\linewidth][c]{
		\begin{subfigure}[t]{.50\textwidth}
			\includegraphics[width=\textwidth]{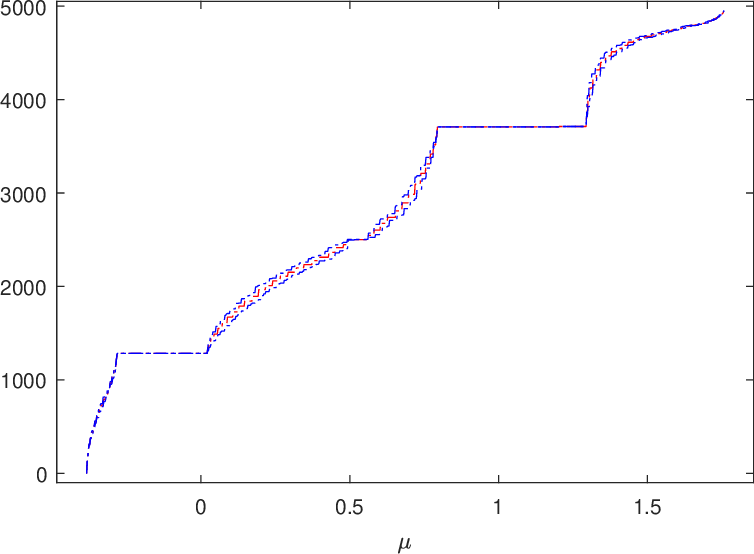}
		\end{subfigure}
		\begin{subfigure}[t]{.50\textwidth}
			\includegraphics[width=\textwidth]{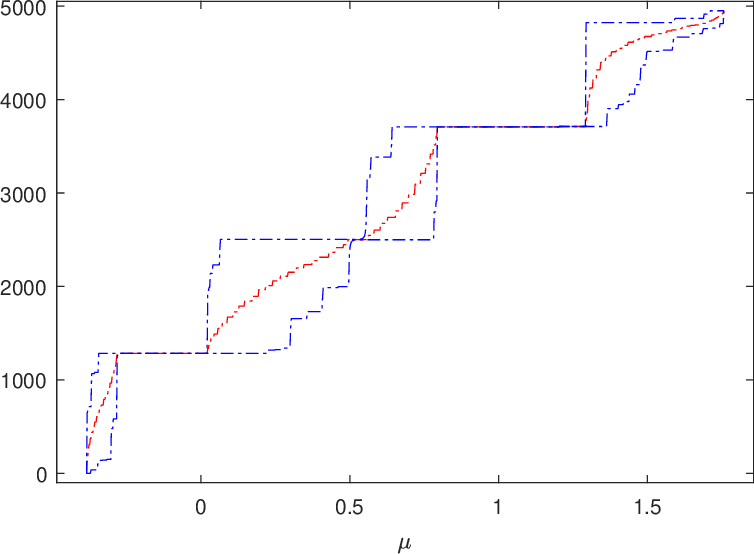}
		\end{subfigure}
		}
	\caption[Approximation of $\trace(P_\mu)$ with Hutchinson's estimator for several $\mu$]{Approximations to $\vec x^T P_\mu \vec x$, and upper and lower bounds obtained with \cref{algorithm:eigenvalue-gap-finder--final} for different $\mu$, with $m = 100$ Lanczos iterations, for the problem in \cref{subsubsec:experiments-hydrogen}. Left: \cref{algorithm:eigenvalue-gap-finder--final} with robust consecutive difference error estimate \cref{eqn:lanczos-upper-lower-robust-safe-estimates}. Right: \cref{algorithm:eigenvalue-gap-finder--final} with a posteriori error bound from \cref{prop:lanczos-a-posteriori-error-bound}.   
	\label{fig:gapfinder-h2-traceest}}
\end{figure}

\begin{table}
	\centering
	\caption[Performance of \cref{algorithm:eigenvalue-gap-finder--final} on the Hamiltonian associated with the H$_2$ chain]{Performance of \cref{algorithm:eigenvalue-gap-finder--final} on the discrete Hamiltonian $A$ associated with a H$_2$ chain, using the robust consecutive difference estimate \cref{eqn:lanczos-upper-lower-robust-safe-estimates} and the a posteriori bound from \cref{prop:lanczos-a-posteriori-error-bound}. The number of Lanczos iterations is set to $m = 100$ and the failure probability to $\delta = 10^{-2}$.}
	\label{table:gapfinder-realmat}
	\begin{tabular}{l|cccS[table-format=2.3]}
		\toprule
		method  & first gap & second gap & third gap & {time (s)} \\
		\midrule
\texttt{eig} & [-0.282, 0.022] & [0.793, 1.204] & [1.204, 1.294] & 15.284 \\
\cref{algorithm:eigenvalue-gap-finder--final} w/~consec.~diff. & [-0.281, 0.021] & [0.794, 1.202] & [1.204, 1.294] & 0.064 \\
\cref{algorithm:eigenvalue-gap-finder--final} w/~\cref{prop:lanczos-a-posteriori-error-bound} & [-0.273, 0.006] & [0.814, 1.202] & [1.204, 1.272] & 0.386 \\
		\bottomrule
	\end{tabular}
\end{table}

\rev{
The gaps found by the two variants of \cref{algorithm:eigenvalue-gap-finder--final} and the exact gaps obtained by diagonalizing~$A$ are shown in \cref{table:gapfinder-realmat}, where we also include the execution times for all the methods. With the chosen parameters, \cref{algorithm:eigenvalue-gap-finder--final} is not able to detect the small gap in the middle of the spectrum, but increasing the number of Lanczos iterations to $m = 300$ allows us to detect that gap as well. The results are consistent with the conclusions drawn from the experiment in \cref{subsubsec:experiments-diraccomb}. 
}

\section{Conclusions}
\label{sec:conclusions}

We have described and analyzed an algorithm for identifying gaps in the spectrum of a real symmetric matrix by simultaneously approximating the traces of spectral projectors associated with different slices of its spectrum, combining Hutchinson's stochastic trace estimator and the Lanczos algorithm. 
We have investigated the behavior of Hutchinson's estimator, and provided a priori and a posteriori error bounds for the Lanczos algorithm. 
This analysis led to an algorithm capable of detecting all gaps larger than a specified threshold with high probability. 
We have also determined that for this problem the best strategy is to use a single random sample vector for Hutchinson's estimator, and utilize all the available matrix-vector products to approximate a single quadratic form via the Lanczos algorithm. 
Our numerical tests demonstrate that the developed algorithm can efficiently and reliably detect gaps, especially when their width is relatively large. 
Within our algorithm we have proposed to use either an a posteriori error bound for the Lanczos approximation error (\cref{prop:lanczos-a-posteriori-error-bound}), or a simple error estimate based on consecutive differences. Although the a posteriori error bound has better theoretical guarantees, the robust version of the error estimate from \cref{example:trace-bound-final} has more accurate results in practice, and it is also computationally cheaper, notably when a large number of values of $\mu$ is used. 
Even when the algorithm fails to detect a gap, it still provides upper and lower bounds for $\vec x^T P_\mu \vec x$ for all considered values of $\mu$, which can be valuable for selecting a starting point for other algorithms, such as those based on $LDL^T$ factorizations.

\section*{Acknowledgements}

\rev{We are grateful to Tyler Chen and two anonymous reviewers for their valuable comments and suggestions.}
The first and third authors are members of the INdAM Research group GNCS (Gruppo Nazionale di Calcolo Scientifico), and their work 
was supported in part by MUR
(Italian Ministry of University and Research) through the PRIN Project
20227PCCKZ (``Low-Rank Structures and Numerical Methods 
in Matrix and Tensor Computations and their Applications''). The first
author also acknowledges partial support from MUR through the 
PNRR MUR Project PE0000023-NQSTI. The work of the second author is funded by the Research Foundation - Flanders (FWO) via the junior postdoctoral fellowship 12A1325N.

\bibliographystyle{siam}
\bibliography{references}

\end{document}